\documentclass[a4paper,10pt]{article}

\usepackage{amsmath}
\usepackage{german, ngerman}
\usepackage[german]{babel}
\usepackage[latin1]{inputenc}
\usepackage{graphics}
\usepackage{verbatim}
\usepackage{listings}
\usepackage{dsfont}
\usepackage{amsmath}
\usepackage{amsfonts}
\usepackage{amssymb}
\usepackage{graphicx}
\usepackage{geometry}
\usepackage{graphics}
\usepackage{url}
\usepackage{framed}
\usepackage{color}
\usepackage{textcomp}

\newtheorem{set2}{Satz}[section]
\newtheorem{theorem}[set2]{Theorem}

\newtheorem{definition}[set2]{Definition}

\newtheorem{lemma}[set2]{Lemma}
\newtheorem{notation[set2]}{Notation}

\newtheorem{proposition}[set2]{Proposition}

\newcommand{\ep}{\hfill{$\blacksquare$}}
\newenvironment{proof}[1][Proof]{\textbf{#1.} }{ \ \\}

\def\XXint#1#2#3{{\setbox0=\hbox{$#1{#2#3}{\int}$}
\vcenter{\hbox{$#2#3$}}\kern-.5\wd0}}

\newcommand{\R}{\mathbb{R}}

\newcommand{\C}{\mathcal}
\newcommand{\ol}{\overline}

\newcommand{\eps}{\varepsilon}
\newcommand{\epsk}{ {\varepsilon_k}}
\newcommand{\be}{\begin{equation*}}
\newcommand{\ee}{\end{equation*}}
\newcommand{\bes}{\begin{equation*}}
\newcommand{\ees}{\end{equation*}}

\newcommand{\iot}{\int_{\Omega_T}}
\newenvironment{new}{\color{ddcyan}}{}
\newcommand{\ben}{\begin{new}}
\newcommand{\enn}{\end{new}}

\allowdisplaybreaks

\definecolor{ddmagenta}{rgb}{0.7,0,1.0}
\definecolor{ddcyan}{rgb}{0,0.1,1.0}
\definecolor{dred}{rgb}{.8,0,0}
\definecolor{ddgreen}{rgb}{0,0.4,0.4}
\definecolor{dgreen}{rgb}{0,0.4,0.2}

\begin{document}

\begin{center}
  {\bf {\Large Modeling and analysis of a phase field system \\
      for damage and phase separation processes\\ 
      in solids\footnote{This project is supported by the DFG Research Center  
        ``Mathematics for Key Technologies''   Matheon in Berlin (Germany) and
        by GNAMPA (Indam, Italy). }\\}}
  \medskip
  Elena BONETTI\footnote{Dipartimento di Matematica F. Casorati - Universit\`a di Pavia
    via Ferrata 1, I-27100 Pavia, Italy, E-Mail: elena.bonetti@unipv.it,
    antonio.segatti@unipv.it. }, 
  Christian HEINEMANN\footnote{Weierstrass Institute for Applied Analysis and
    Stochastics, Mohrenstrasse 39, 10117 Berlin, Germany, E-Mail:
    {christian.heinemann@wias-berlin.de}, {christiane.kraus@wias-berlin.de}.}, Christiane KRAUS$^3$ and Antonio
  SEGATTI$^2$ 
\end{center}
\\[2mm]
{\it Keywords:} {Cahn-Hilliard system, phase separation, elliptic-parabolic systems, doubly nonlinear
  differential inclusions, complete damage, existence results, 
  energetic solutions, weak solutions,  linear elasticity, rate-dependent systems.  \\[4mm]
}
\begin{center}
  {\bf Abstract}
\end{center}
In this work, we analytically investigate a multi-component system for
describing phase separation and damage processes in solids. 
The model consists of a parabolic diffusion equation of fourth order for the 
concentration coupled with an elliptic system with material dependent
coefficients for the strain tensor and a doubly nonlinear differential
inclusion for the damage function. The main aim of this paper is to show
existence of weak solutions for the introduced model, where, in contrast to 
existing damage models in the literature, different elastic properties of
damaged and undamaged material are regarded. To prove existence of 
weak solutions for the introduced model, we start with an approximation
system. Then, by passing to the limit, existence results of weak solutions 
for the proposed model are obtained via suitable variational techniques.

\section{Introduction}
The ongoing miniaturization in the area of micro-electronics leads to higher demands on
strength and lifetime of the materials, while the structural size is
continuously being reduced. Materials, which enable the functionality 
of technical products, change the microstructure over time. 
Phase separation, coarsening phenomena and damage processes take
place. The complete failure of electronic  devices like motherboards or mobile
phones often results from micro--cracks in solder joints.  Therefore, the
knowledge of the mechanisms inducing phase separation, coarsening and damage 
phenomena is of great importance for technological applications. A uniform
distribution of the original materials is aimed to guarantee evenly
distributed material properties of the sample. For instance, mechanical
properties, such as the strength and the stability of the material, depend on 
how finely regions of the original materials are mixed. The control of the evolution of the
microstructure  and, therefore, of the lifetime of materials
relies on the ability to understand phase separation, coarsening and
damage  processes. Hence, a major aim is to develop reliable mathematical 
models for describing such effects. 

Phase separation and coarsening phenomena are usually described by 
phase--field models of Cahn--Hilliard type. The evolution is
modeled by  a parabolic diffusion equation for the phase fractions.  
To include elastic effects, resulting from stresses caused by different
elastic properties of the phases, Cahn-Hilliard systems are coupled 
with an elliptic equation, describing quasi-static balance of  
forces. Such coupled Cahn-Hilliard systems with elasticity are also called
Cahn-Larch\'e systems.  Since in general the mobility,
stiffness and surface tension coefficients depend on the phases (see 
for instance \cite{BDM07} and \cite{BDDM09} for the explicit structure deduced by the embedded atom
method), the mathematical analysis of the coupled problem is very
complex. Existence results were derived for special cases in \cite{GarckeHabil,Carrive00,Pawlow} 
(constant mobility, stiffness and surface tension coefficients), 
in \cite{Bonetti02} (concentration dependent mobility,  two
space dimensions), \cite{SP12, SP13}  (concentration dependent surface
tension and nonlinear diffusion) and in \cite{Pawlow08} in an abstract measure-valued
setting (concentration dependent mobility and surface tension tensors).  
For numerical results and simulations we refer e.g. to \cite{Wei02,Mer05,BM2010}.

From a microscopic point of view, damage behavior originates from breaking atomic links in the material
whereas a macroscopic theory may specify damage in the isotropic case by a
scalar variable related to the proportion of damaged bonds in the
micro-structure of the material with respect to the undamaged ones.  
According to the latter perspective, phase-field models are quite
common to model smooth transitions between damaged and undamaged
material states. Such phase-field models have been mainly
investigated for incomplete damage which means that damaged material cannot 
loose all its elastic energy. 

A first local in time existence result for a 3D damage  model has been introduced in \cite{BS04},
where irreversibility of the damage evolution is accounted for. 
Damage for viscoelastic materials, in which also viscosity degenerates during
the damage process,  is investigated  in
\cite{BSS05}.  Damage models are also 
analytically investigated in \cite{MT10, KRZ11} and, there, existence,
uniqueness and regularity properties are shown.  These models do not account
for temperature effect. A local in time existence result for a complete
dissipative damage model with the evolving of temperature can be found in
\cite{BB}.  A {\it  coupled} system describing incomplete {\it damage}, 
linear elasticity and {\it phase separation} appeared
firstly in \cite{HK10a, HK10b}.  There, existence of weak solutions has 
been proven under mild assumptions, where, for instance, the stiffness 
tensor may be material-dependent and the chemical free energy may be 
of polynomial or logarithmic type. 
All these works are based on the gradient-of-damage model proposed by
Fr\'emond and Nedjar \cite{FN96} (see also \cite{Fre02}) which
describes damage as a result from microscopic movements in the solid.
The distinction between a balance law for the microscopic forces and
constitutive relations of the material yield a satisfying derivation
of an evolution law for the damage propagation from the physical point of view.
In particular, the gradient of the damage variable enters the resulting equations
and serves as a regularization term for the mathematical analysis as well as it ensures the structural size effect.
Internal constraints are ensured by the presence of non-smooth operators
(subdifferential operators) in the evolution equations. Hence, in the case that the evolution of the damage is assumed to be uni-directional, i.e. the
damage process is irreversible,
the microforce balance law becomes a doubly-nonlinear differential inclusion. 

For a non-gradient approach of damage models for brittle materials we refer to
\cite{FG06,GL09, Bab11}. There, the damage variable $z$ takes on two distinct values,
i.e. $\{0,1\}$, in contrast to the gradient approach, where $z \in [0,1]$. In addition,
the mechanical properties  are described in \cite{FG06, GL09, Bab11} differently. They 
choose a $z$-mixture of a linearly elastic strong and weak material
with two different elasticity tensors. A non-gradient  model for incomplete
damage in the framework of Young measures is considered in \cite{FKS12}.

Damage modeling is an active field in the engineering community since the
1970s. We do not actually detail literature.  For some recent works we refer to \cite{Car86, NPO94, Mie95,
  MK00, MS01, Fre02, LD05, Gee07, VSL11}. A variational 
approach to fracture and crack propagation models can be found for
instance in  \cite{ BFM08, CFM09, CFM10, Neg10, LT11}. 

The reason why incomplete damage models are more feasible for mathematical investigations is that
a coercivity assumption on the elastic energy prevents the material
from a complete degeneration. Typically, the following form is chosen:
\begin{equation}
\label{eq:incomp-phi1}
W^\mathrm{el}(e,z)= \frac{1}{2} (\Phi(z) + \delta) \, {\mathbb C} e: e,
\qquad \delta >0 \, \text{ small},
\end{equation}
where $\Phi:[0,1] \to \R_+$ is a continuous and monotonically increasing
function with $\Phi(0)=0$. The symbol ${\mathbb C}$ denotes the stiffness
tensor and $e$ is the strain tensor. 

Dropping $\delta>0$ in \eqref{eq:incomp-phi1} may lead to serious troubles.
However, in the case of viscoelastic materials, the inertia terms
circumvent this kind of problem in the sense that
the deformation field still exists on the whole domain accompanied with a loss 
of spatial regularity (cf. \cite{RR12}). Unfortunately, this result cannot be
expected in the case of quasi-static
mechanical equilibrium (see for instance \cite{BMR09}). Mathematical works dealing
with complete models and covering global-in-time existence are rare and are mainly
focused on purely \textit{rate-independent systems} \cite{Mielke06, BMR09,
  Mielke10, Mie11} by using $\Gamma$-convergence techniques to recover
energetic properties in the limit. Very recently, global-in-time existence results are also
obtained for \textit{rate-dependent systems} in \cite{HK12, HK13b} by considering the damage
process on a time-dependent domain. Alternatively, in \cite{BFS13} the problem
of understanding complete damage is tackled using some defect measures
which are conjectured to concentrate on the (complete) damaged portions 
of the material. This theoretical prediction is supported by numerical simulations.

The main aim of this work is to show existence of weak solutions of  
a unified model for phase separation and damage processes,  where, 
in contrast to the existing incomplete damage models in the literature
\cite{MT10, KRZ11, HK10a, HK10b} or local in time damage evolution 
\cite{BS04,BSS05}, different elastic properties of 
damaged and undamaged material are regarded. 
More precisely, we choose an elastic energy density
${W}^\mathrm{el}$ of the form 
\begin{equation} \label{eq:W1}
    W^\mathrm{el}(e,c,z) =  \Phi(z) \, {W}_1^\mathrm{el}(e,c)   +
  \big(1- \Phi(z) \big)   {W}_2^\mathrm{el}(e,c)   ,
\end{equation} 
where $\Phi:[0,1] \to [0,1]$ is a continuously differentiable and monotonically
increasing function with $\Phi(0)=\Phi'(0)=0$, $ \Phi(1)=1$, $W_1^\mathrm{el} \ge
W_2^\mathrm{el} $ and $c$ is the concentration field. This means that for undamaged material 
the elastic energy density $W_1^\mathrm{el}$ is stored, whereas in the
completely damaged case $z=0$ the energy $W_2^\mathrm{el}$ is stored. For the elastic energy 
$W_1^\mathrm{el}$ we assume an $H^1$-coercivity condition for $u$ 
and for  $W_2^\mathrm{el}$ a weaker $W^{1,p}$-coercivity condition, $1<p<2$. 

Our highly nonlinear model covers the intermediate case between incomplete and complete 
damage which takes care for different deformation properties of damaged 
and undamaged material. It consists of a parabolic diffusion equation of fourth order 
for the concentration coupled with an elliptic system with material dependent
coefficients for the strain tensor and a doubly nonlinear differential
inclusion for the damage function, see Definition ($S_0$) on page \pageref{page:system}.

The paper is organized as follows: In Section \ref{sec:model}, we 
start with introducing the model formally and stating the notation and
assumptions. Then, we  introduce an appropriate notion of weak solutions 
for our introduced system in Subsection \ref{section:weakFormulation}. 
To handle the differential inclusion rigorously, we adapt the 
concept of weak solutions which has been proposed in \cite{HK10a} 
for phase separation systems coupled with rate-dependent damage 
processes. The main result is stated in Subsection \ref{sec:main}.  Section
\ref{sec:proof} is devoted to the existence proof of the proposed model. 
The proof is based on an approximation-a priori estimates-passage to
the limit procedure. In particular, the limit analysis relies on the monotone
structure of the system.

\section{Modeling \label{sec:model}}
We consider an $N$-component alloy occupying a bounded Lipschitz domain
$\Omega\subseteq\mathbb R^3$. To account for deformation,  phase separation and
damage processes, a state of the system at a fixed time point is specified by the triple $q=(u,c,z)$.
The displacement field $u:\Omega\rightarrow\mathbb R^3$ determines the current position $x+u(x)$ of an
undeformed material point $x$.
Throughout this paper, we will work with the linearized strain tensor $e(u)=\frac 12(\nabla u+(\nabla u)^T)$,
which is an adequate assumption only when small strains occur in the
material. However, this
assumption is justified for phase separation processes in alloys since the deformation usually has a
small gradient.
The vector-valued function $c:\Omega\rightarrow\mathbb R^N$ describes the
chemical concentration of the $N$-components, which satisfies the normalized
condition $\sum_{j=1}^N c_j=1$ in $\Omega$. To
account for damage effects, we choose a scalar damage
variable $z:\Omega\rightarrow\mathbb R$, which models the 
reduction of the effective volume of the material due to void
nucleation, growth, and coalescence.  The damage process is
modeled unidirectional, i.e.~damage may only
increase. In particular, self-healing processes in the material are
forbidden. No damage at a material point $x \in \Omega$ is described by
$z(x)=1$, whereas $z(x)=0$ stands for a completely damaged material point $x
\in \Omega$. 

\subsection{Energies and evolutionary equations}
Here, we qualify our model formally and postpone a rigorous
treatment to Section \ref{section:weakFormulation}.
The presented model is based on two functionals, i.e.~a
generalized Ginzburg-Landau free energy functional $\mathcal E$ and a damage
pseudo-dissipation potential $\mathcal R$  
(in the sense by Moreau). The free energy density $\varphi$ of 
the system is given by
\begin{equation}
  \label{eq:free_energy}  
  \varphi(e(u),c,\nabla c,z,\nabla z):= \frac{\gamma}{2} \mathbf\Gamma\nabla
  c:\nabla c   +\frac{\delta}{2}|\nabla
  z|^2+W^\mathrm{ch}(c)+W^\mathrm{el}(e(u),c,z),\qquad\gamma,\delta>0,
\end{equation}
where the gradient terms  penalize spatial changes of the
variables $c$ and $z$, $W^\mathrm{ch}$ denotes the chemical energy density and
$W_\mathrm{el}$ is the elastically stored energy density
accounting for elastic deformations and damage effects. 
For simplicity of notation, we set $\gamma=\delta=1$.

The {\it chemical free energy density} $W^\mathrm{ch}$ depends
on temperature, which is convex above a critical temperature
value and non-convex below. Therefore, if an alloy is cooled down
below the critical temperature, spinodal decomposition and
coarsening phenomena occur due to the several local minimizers
of  $W^\mathrm{ch}$. We assume that the chemical energy is of 
polynomial type. More precisely, we need the assumptions
\eqref{eqn:growthEst7c}-\eqref{eqn:growthEst8}
of Section \ref{section:AssRes} for a rigorous treatment. 

The {\it elastically stored energy density} 
${W}^\mathrm{el}_1$ in \eqref{eq:W1} due to stresses and strains, which occur
in the material, is typically of quadratic form, i.e.  
\begin{equation}
\label{eq1:incomp-phi1}
  {W}^\mathrm{el}_1(e(u),c) = 
  \frac{1}{2} \big(e(u) - e^*(c)\big) : \mathbb C(c) \big(e(u) -e^*(c)\big). 
\end{equation}
Here, $e^*(c)$ denotes the {\it eigenstrain}, which is usually linear 
in $c$, and  $\mathbb C(c)\in\mathcal L(\mathbb R_\mathrm{sym}^{n\times n})$ 
is a fourth order stiffness tensor, which may depend on the concentration.
The stiffness tensor is assumed to be symmetric and positive definite. 
Note that we are not restricted to homogeneous elasticity.  

To incorporate the effect of damage on the elastic response of
the material, we choose an elastic energy density
${W}^\mathrm{el}$ of the form \eqref{eq:W1}, i.e. 
\begin{equation}
  \label{eq:Wel}
  W^\mathrm{el}(e(u),c,z) =  \Phi(z) \, {W}_1^\mathrm{el}(e(u),c)   +
  \big(1- \Phi(z) \big)   {W}_2^\mathrm{el}(e(u),c)   ,
\end{equation} 
where $\Phi:[0,1] \to \R_+$ is a continuously differentiable and monotonically
increasing function with $\Phi(0)=\Phi'(0)=0$, $\Phi(1)=1$ and $W_1^\mathrm{el} \ge
W_2 ^\mathrm{el} $. This means that in the undamaged case the material 
accumulates the elastic energy density $W_1^\mathrm{el}$, whereas in the
completely damaged case only the lower energy $W_2^\mathrm{el}$ is stored. 
Hence, in particular, different elastic properties of damaged and undamaged 
material can be modeled. 

We assume that $W_1^\mathrm{el}$ is of quadratic growth in $e$, whereas 
$W_2^\mathrm{el}$ only has to satisfy a lower $p$-growth condition,  $1 < p < 2$.
This means that the displacement field for damaged material only need to be 
an element of $L^p(\Omega)$, $1 < p < 2$. The complete growth conditions 
for $W^\mathrm{el}$ can be found in Section \ref{section:AssRes}. \\[2mm]
The overall free energy $\mathcal E$ of Ginzburg-Landau type
has the following structure:
\begin{equation}
  \label{eqn:EnergyTyp1}
  \begin{split}
    &\mathcal E(u,c,z):=\tilde{\mathcal E}(u,c,z)+\int_\Omega I_{[0,\infty)}(z)\,\mathrm dx,\\
      &\tilde{\mathcal E}(u,c,z):=\int_\Omega \varphi(e(u),c,\nabla c,z,\nabla z)\,\mathrm dx.
  \end{split}
\end{equation}
Here, $ I_{[0,\infty)}$ signifies the indicator function of the subset 
$[0,\infty)\subseteq\mathbb R$, i.e. $ I_{[0,\infty)}(x)=0$ for $x \in
      [0,\infty)$ and  $ I_{[0,\infty)}(x) =\infty$ for $(-\infty,0)$. 

We assume that the energy dissipation for the damage process is triggered by a 
dissipation potential $\mathcal R$ of the form
\begin{equation}
  \label{eqn:EnergyTyp2}
  \begin{split}
    &\mathcal R(\dot z):=\tilde{\mathcal R}(\dot z)
    +\int_\Omega I_{(-\infty,0]}(\dot z)\,\mathrm dx,\\
	&\tilde{\mathcal R}(\dot z):=\int_\Omega
        \Big(-\alpha\dot z+\frac 12\beta\dot z^2 \Big)\,\mathrm dx
	\, \, \text{ for } \alpha \ge0\text{ and }\beta>0.
  \end{split}
\end{equation}
Due to $\beta>0$, the dissipation potential is referred to as
\textit{rate-dependent}. In the case $\beta=0$, which is not
considered in this work, $\mathcal R$ is called \textit{rate-independent}. 
We refer for rate-independent processes to \cite{Efendiev06, Mielke99,
  Mielke06, Mielke10, Roubicek10} 
and in particular to  \cite{Mielke05} for a survey.

The governing evolutionary equations for a system state $q=(u,c,z)$ can be
expressed by virtue of the functionals \eqref{eqn:EnergyTyp1} and
\eqref{eqn:EnergyTyp2}. The evolution is driven by the following
elliptic-parabolic system of differential equations and differential inclusion: 
\label{page:system}
\begin{align} 
  \left.
  \begin{cases} 
    \textit{Diffusion}:\hspace{2.9cm}\partial_t c &= \mathrm{div} (\mathbb M  \nabla  w )\\
    \hspace{4.7cm} w&=\mathbb P \big((-\mathrm{div}(\mathbf\Gamma\nabla c)
    +W_{,c}^\mathrm{ch}(c)+W_{,c}^\mathrm{el}(e(u),c,z) \big)\\
    \textit{Balance of forces}: \hspace{1.35cm} \mathrm{div} \, \sigma &=f \\
    \textit{Damage evolution}:\hspace{1.8cm}  0 & \in \partial_z\mathcal
    E(u,c,z)+\partial_{\dot z}\mathcal R(\partial_t z)
  \end{cases}
  \right \}
  \label{eqn:unifyingModelClassical}
  \tag{{$S_0$}}
\end{align}
where $\sigma=\sigma(e,c,z):=\partial_e \varphi(e,c,\nabla c,z,\nabla z)$ 
denotes the Cauchy stress tensor, $w$ is the chemical potential given by
$w=w(u,c,z):=\partial_c\varphi(e,c,\nabla c,z,\nabla z)
-\mathrm{div}(\partial_{\nabla c}\varphi(e,c,\nabla c,z,\nabla z))$ 
and $-f$ stands for the exterior volume force applied to the body.
The matrix $\mathbb P$  denotes the orthogonal projection of $\R^N$ onto the 
tangent space $T\Sigma=\big\{x\in\mathbb R^N\,|\,\sum_{k=1}^N x_k=0\big\}$
of the affine plane $\Sigma:=\big\{x\in\mathbb R^N\,\big|\,\sum_{l=1}^N
x_k=1\big\}$. The diffusion equation is a fourth order quasi-linear parabolic equation of
Cahn-Hilliard type and models phase separation processes for the concentration
$c$ while the balance of forces is described by an elliptic equation  for
$u$. The doubly nonlinear differential inclusion specifies the flow rule of the
damage profile according to the constraints $0\leq z\leq 1$ and $\partial_t z\leq 0$ (in space and time). Actually,
we have $z\leq 1$ combining the two constraints   $z\geq0$ and $\partial_t
z\leq 0$ (irreversible damage), once the initial datum is lower than 1. 
The inclusion has to be read in terms of generalized subdifferentials.

We need to impose some restrictions on the mobility matrix $\mathbb
M$.  We assume that  $\mathbb M$ is symmetric and positive definite on
the tangent space $T\Sigma$. In addition, due to the constraint
$\sum_{k=1}^N c_k=1$, $\mathbb M$ has to satisfy the property
$\sum_{l=1}^N \mathbb M_{kl}=0$ for all $k=1,\ldots,N$. Note, that
$\mathbb M= \mathbb M \, \mathbb P$. The gradient tensor $\Gamma$ 
is assumed to be symmetric and positive definite. 

Let $D\subset\partial \Omega$ with  $\mathcal
H^{n-1}(D)>0$ ($\mathcal H^n$: $n$-dimensional
Hausdorff measure) denote the portion of the boundary $\partial\Omega$
on which we prescribe Dirichlet boundary conditions. We set
$ D_T:=(0,T)\times D$ and $ (\partial\Omega)_T:=(0,T)\times
\partial\Omega$.
The initial-boundary conditions of our system are summarized as follows:\\[1mm]
\textit{Initial conditions}
\begin{align*} \label{eqn:unifyingModelIBC} 
  & c(0) =c^0  \text{ a.e. in }\Omega \quad \text{ and }  \quad
  c^0\in \Sigma  \text{ a.e. in }\Omega, \\[-3mm]
  & \\[-3mm]
  &0     \le z(0)=z^0 \le 1 \text{ a.e. in }\Omega.\\[-2mm] 
  \tag{IBC} \\[-6mm]
  \intertext{\textit{Boundary conditions}}   \\[-8mm]
  & u=b \text{ on } D_T, \quad \sigma\cdot\overrightarrow{\nu}=0 \text{ on }(\partial\Omega)_T
  \setminus D_T. \\
  & \nabla z\cdot \overrightarrow{\nu}=0\text{ on }(\partial\Omega)_T, \quad 
  \mathbf\Gamma\nabla c\cdot\overrightarrow{\nu}=0\text{ on }(\partial\Omega)_T, \qquad 
  \mathbb M\nabla w\cdot\overrightarrow{\nu}=0 \text{ on }(\partial\Omega)_T, \\[1mm]
  & \text{where $\overrightarrow{\nu}$} \text{stands for the unit normal on $\partial\Omega$ pointing outward and
    $b$ is the boundary value}  \\[-1mm]
  & \text{function on  the Dirichlet boundary $D$, which can be suitably extended to a function on
    $\ol{\Omega_T}$.}  
\end{align*}
To show existence of weak solutions for the system
\eqref{eqn:unifyingModelClassical}, we first consider a regularized version 
for the displacement field: \\[2mm]
{\it Regularized energy} 
\begin{align*}
  &\tilde{\mathcal E}_\varepsilon(u,c,z):=
  \int_\Omega  \Big( \frac{1}{2}\mathbf\Gamma\nabla c:\nabla c
  +\frac{1}{2}|\nabla z|^2+W^{\mathrm{ch}}(c)+W^\mathrm{el}(e,c,z)
  +\frac \varepsilon4|\nabla u|^4  \Big) \,\mathrm dx,\\
  &\mathcal E_\varepsilon(u,c,z):=\tilde{\mathcal E}_\varepsilon(u,c,z)+\int_\Omega I_{[0,\infty)}(z)\,\mathrm dx.
\end{align*}
{\it Evolution system}
\begin{align} 
  \! \! \!   \! \! \!   \left.
  \begin{cases} 
    \textit{Diffusion}:\hspace{3.7cm}\partial_t c & \! \! \!  \!=
    \mathrm{div} (\mathbb M  \nabla  w )\\
    \hspace{5.5cm} w&\!\! \!  \!=\mathbb P(-\mathrm{div}(\mathbf\Gamma\nabla
    c)+W_{,c}^\mathrm{ch}(c)+W_{,c}^\mathrm{el}(e(u),c,z))\\
    \textit{Balance of forces}: \hspace{0.15cm} \mathrm{div} \sigma +
    \varepsilon\mathrm{div}(|\nabla u|^2\nabla u)&\! \!  \!\!=f \\
    \textit{Damage evolution}:\hspace{2.7cm}  0 &\! \!  \! \! \in \partial_z\mathcal
    E_\eps(u,c,z)+\partial_{\dot z}\mathcal R(\partial_t z) 
  \end{cases}\! 
  \right \}
  \label{eqn:regUnifyingModelClassical}
  \tag{{$S_\eps$}}
\end{align}
{\it Initial-boundary conditions} 
\begin{align}
  \text{\eqref{eqn:unifyingModelIBC} with }(\sigma+\varepsilon|\nabla u|^2\nabla u)\cdot\overrightarrow{\nu}=0 \quad 
  \text{ instead of } \quad \sigma\cdot\overrightarrow{\nu}=0   \quad \text{on }(\partial\Omega)_T.
  \label{eqn:regUnifyingModelIBC}
  \tag{IBC$_\varepsilon$}
\end{align}
\subsection{Notation}
The notation, we will use throughout this paper, is collected in the
following list.\\[2mm]
\textit{Spaces and sets.}\vspace*{0.4em}\\
\begin{tabular}{p{6em}l}
  $W^{1,r}(\Omega;\mathbb R^n)$ & standard Sobolev space\\
\end{tabular}\\
\begin{tabular}{p{6em}l}
  $W^{1,r}_+(\Omega)$ &functions of $W^{1,r}(\Omega)$ which are non-negative almost everywhere\vspace{0.2em}\\
\end{tabular}\\
\begin{tabular}{p{6em}l}
  $W^{1,r}_-(\Omega)$ &functions of $W^{1,r}(\Omega)$ which are non-positive almost everywhere\vspace{0.2em}\\
\end{tabular}\\
\begin{tabular}{p{6em}l}
  $W^{1,r}_D(\Omega;\mathbb R^n)$ &functions of $W^{1,r}(\Omega;\mathbb R^n)$ which vanish
  on $D\subseteq\partial\Omega$ in the sense of traces\vspace{0.2em}\\
\end{tabular}\\
\begin{tabular}{p{6em}l}
  $G_T$ & $(0,T)\times G$\vspace{0.2em}\\
\end{tabular}\\
\begin{tabular}{p{6em}l}
  $\R_+$ & $\{ x \in \R: x  \ge 0 \}$
\end{tabular} \\[5mm]
\textit{Functions, operations and measures.}\vspace*{0.4em}\\
\begin{tabular}{p{6em}l}
  $I_M$ & indicator function of a subset $M\subseteq X$\vspace{0.2em}\\
\end{tabular}\\
\begin{tabular}{p{6em}l}
  $W_{,e}$ &classical partial derivative of a function $W$ with respect to the variable $e$\vspace{0.2em}\\
\end{tabular}\\
\begin{tabular}{p{6em}l}
  $\langle g^\star,f\rangle$ &dual pairing of $g^\star\in (W^{1,r}(\Omega;\R^n))^\star$ and $f\in W^{1,r}(\Omega;\R^n)$\vspace{0.2em}\\
\end{tabular}\\
\begin{tabular}{p{6em}l}
  $\mathrm d E$ &G\^{a}teaux differential of $E$\vspace{0.2em}\\
\end{tabular}\\
\begin{tabular}{p{6em}l}
  $p^\star$ &Sobolev critical exponent $\frac{np}{n-p}$ for $n>p$\vspace{0.2em}\\
\end{tabular}\\
\begin{tabular}{p{6em}l}
  $\mathcal H^n$ &Hausdorff measure of dimension $n$\vspace{0.2em}\\
\end{tabular}\\
\begin{tabular}{p{6em}l}
  $\mathcal L^n$ &Lebesgue measure of dimension $n$
\end{tabular}\\
\begin{tabular}{p{6em}l}
  $\Sigma$ & $\big\{x\in\mathbb R^N\,\big|\,\sum_{l=1}^N x_k=1\big\}$
\end{tabular} \\ 
\begin{tabular}{p{6em}l}
  $T\Sigma$ & $\big\{x\in\mathbb R^N\,\big|\,\sum_{l=1}^N x_k=0\big\}$
\end{tabular}\\

\subsection{Assumptions  \label{section:AssRes}}
The general setting, the growth assumptions and the assumptions on the
coefficient tensors which are mandatory for the existence theorem are summarized below.
\begin{enumerate}
  \renewcommand{\labelenumi}{(\roman{enumi})}
\item \textit{Setting}
  \begin{align*}
    \begin{aligned}
      &\text{Space dimension}\hspace{4em}&&n \in \mathbb{N},\hspace{22.5em}\\
      &\text{Components in the alloy} &&N\in\mathbb N\text{ with }N\geq 2,\\
      &\text{Regularization exponent} && 1 < p < 2,\\
      &\text{Conjugate exponent} && p'=\frac{p}{p-1},\\
      & \text{Growth exponent } && s < \frac{n(p-1)}{n-p} ,\\
      &\text{Viscosity factors} &&\alpha,\beta>0,\\
      &\text{Domain} &&\Omega\subseteq\mathbb R^n\text{ bounded Lipschitz domain,}\\
      &\text{Dirichlet boundary} &&D\subseteq\partial\Omega\text{ with }\mathcal H^{n-1}(D)>0,\\
      &\text{Time interval} &&[0,T]\text{ with }T>0,\\
      &\text{External volume force} &&  
      f  \in W^{1,1} (0,T; L^{p'}(\Omega;\R^n)) \text{
        with  } f(0)=f^0 \in L^{p'}(\Omega;\R^n) ,\\
      &\text{Constant} &&  C > 0 \text{ (context dependent)}
    \end{aligned}
  \end{align*}
\item \textit{Energy densities}
  \begin{align}
    & \Phi \in \C C^1([0,1];[0,1])  \text{ monotonically
      increasing with}  \notag \\ &\Phi(0)=\Phi'(0)=0 \text{ and }  \Phi(1)=1\notag .\\
    \text{Elastic energy density } W_1^\mathrm{el} \qquad
    & W_1^\mathrm{el} \in \C  C^1( \R^{n \times n}\times \R^N; \R) \text{ with} \notag\\	
    \label{eqn:W1growthEst1} \tag{A1}
    &W_1^\mathrm{el}(e,c)=W_1^\mathrm{el}(e^t,c),\\
    \tag{A2}
    \label{eqn:W1growthEst2a}
    &|W_1^\mathrm{el}(e,c)|\leq C
    (|e|^2+|c|^{ 2}+1), \\
    \tag{A3}
    \label{eqn:W1growthEst3}
    &C |e_1-e_2|^2\leq
    \big(W_{1,e}^\mathrm{el}(e_1,c)-W_{1,e}^\mathrm{el}(e_2,c)
    \big):(e_1-e_2),\\
    \tag{A4}
    \label{eqn:W1growthEst4}
    &|W_{1,e}^\mathrm{el}(e_1+e_2,c)|\leq
    C(W_1^\mathrm{el}(e_1,c)+ { |e_2|}+1),\\
    \tag{A5}  \label{eqn:W1growthEst5a}
    &|W_{1,c}^\mathrm{el}(e,c)|\leq C (|e|^2+|c|^{2}+1) \\
    &\text{for any } e_1, e_2 \in  \R^{n \times n}_\text{sym} \text{ and } c \in \Sigma, \notag\\
    \label{eqn:W1growthEst4a} \notag
    &  h_c(\cdot)   = W_{1,e}^\mathrm{el}(\cdot,c) - W_{1,e}^\mathrm{el}(0,c)
    \text{ is positively}\\ 
    & \text{$1$-homogeneous, i.e.} \notag\\
    \tag{A6} &    h_c(\lambda e) =  \lambda h_c(e)
    \text{ for any $\lambda >0$ and all $e \in \R^{n \times n}_\text{sym}$.}  \notag \\
    \text{Elastic energy density } W_2^\mathrm{el} \qquad
    &W_2^\mathrm{el}\in \C C^1(\mathbb R^{n\times n}\times\mathbb R^N ;\mathbb R)\text{ with}\notag\\
    \tag{A7} \label{eqn:growthEst1}   &W_2^\mathrm{el}(e,c)=W_2^\mathrm{el}(e^t,c),\\
    \tag{A8} \label{eqn:growthEst2}   &W_2^\mathrm{el}(e,c)\le W_1^\mathrm{el}(e,c),\\	
    &|W_2^\mathrm{el}(e,c)|\leq C  (|e|^p+|c|^{ s}+1), \tag{A9} \label{eqn:growthEst2a}\\
    \tag{A10}  \label{eqn:growthEst3}
    &C |e_1-e_2|^p  \leq (W_{2,e}^\mathrm{el}(e_1,c)-W_{2,e}^\mathrm{el}(e_2,c))\!:\!(e_1-e_2),\\
    \tag{A11}
    \label{eqn:growthEst4} &|W_{2,e}^\mathrm{el}(e_1+e_2,c)|\leq C(W_2^\mathrm{el}(e_1,c)+ { |e_2|^{p-1}}+1),\\
    \tag{A12} \label{eqn:growthEst5a}
    &|W_{2,c}^\mathrm{el}(e,c)|\leq C (|e|^p+|c|^{s}+1)\\
    \notag
    &\text{for any } e_1, e_2 \in  \R^{n \times  n}_\text{sym} \text{ and } c \in \Sigma.
    \notag\\
    \text{Chemical energy
      density} \qquad \, \, \, \, &W^\mathrm{ch}
    \in \C C^1(\mathbb R^N;\mathbb R)\text{ with }W^\mathrm{ch}\geq -C, 
    \tag{A13} \label{eqn:growthEst7c}\\
    \tag{A14}
    \label{eqn:growthEst8}
    &|W_{,c}^\mathrm{ch}(c)|\leq
    C(|c|^{2^\star/2}+1)\\
    &\text{for any } c \in \Sigma.
    \notag
  \end{align}
\item \textit{Tensors}
  \begin{align*}
    \begin{aligned}
      &\text{Mobility tensor\hspace{3.7em}}
      &&\mathbb M\in\mathbb R^{N\times N}\text{
        symmetric and positive definite on }T\Sigma\text{ and }\hspace{1.0em}\\
      &&&\sum_{l=1}^N \mathbb M_{kl}=0\text{ for all }k=1,\ldots,N.\\
      &\text{Energy gradient tensor} &&
      \mathbf\Gamma\in\mathcal L(\mathbb R^{N\times n};\mathbb
      R^{N\times n})\text{ constant,  symmetric and positive definite}\\
      &&&\text{fourth order tensor.}
    \end{aligned}
  \end{align*}
\end{enumerate}
Note that \eqref{eqn:W1growthEst3}, \eqref{eqn:W1growthEst4},
\eqref{eqn:growthEst3} and \eqref{eqn:growthEst4} imply 
the growth conditions 
\begin{equation}
\begin{split}
\label{eq:cons_ass}
W_1(e,c) \ge C_1 |e|^2 -  C_2(|c|^4 + 1) \qquad \text{and} \qquad 
W_2^\mathrm{el}(e,c) \ge C_1 |e|^p
-  C_2(|c|^{sp'} + 1) 
\end{split}
\end{equation}
for all $c \in \Sigma$ and $ e \in \R^{n \times
                                    n}_\text{sym}$.

Let us point out that the above properties are satisfied in the case we choose 
$W_1$ as in \eqref{eq1:incomp-phi1} and for $W_2$ we may take, for instance, 
 $$W^\mathrm{el}_2(c,e(u))= \frac{1}{2} \big( 
(e(u) -  \hat{e}(c)): 
\hat{\mathbb {C} }(c)
(e(u)- {\hat{e}(c)}) 
\big)^{p/2} -C, \, \, \, \,1 < p < 2, $$
where  $C\ge 0$ is some constant. 
\subsection{Weak formulation  \label{section:weakFormulation} }

In this subsection, we state the notion of weak solutions for our proposed
system and its regularized version. We use the concept of weak solutions introduced in
\cite{HK10a} which consists of an energy inequality
and a variational inequality for the doubly nonlinear differential inclusion.

The next Proposition (see \cite{HK10a, HK10b})
collects the basic properties of this concept of weak solution. In particular, note
that the sole condition ($ii$) is weaker than the usual 
variational inequality that characterizes the doubly nonlinear inclusion $(i)$. 
\begin{proposition}
  \label{prop:energeticFormulation}
Let $(u,c,w,z)\in \C C^2(\Omega_T;\mathbb R^n\times\mathbb R^N\times\mathbb
R^N\times\mathbb R)$ satisfy the diffusion equation and the equation of
balance of forces of $(S_0)$  with initial-boundary
conditions \eqref{eqn:unifyingModelIBC}.
Then the following two properties are equivalent for all $t\in[0,T]$:
\begin{itemize}
\item[(i)]
  $0 \in \partial_z\mathcal E(u(t),c(t),z(t))+\partial_{\dot z}\mathcal R(\dot z(t))$, 
\item[(ii)]
  Energy inequality
  \begin{align*}
    &\mathcal E(u(t),c(t),z(t))
    +\int_{0}^{t}\langle\mathrm d_{\dot z}\tilde{\mathcal R}(\partial_t z),\partial_t z\rangle\,\mathrm ds
    +\int_{\Omega_t} \mathcal \nabla w : {\mathbb M} \nabla w \,\mathrm dx
    \mathrm  ds - \int_{\Omega} f(t) \cdot u(t) \,\mathrm dx\mathrm dx
    \notag\\
    &\qquad\qquad\leq
    \mathcal
    E(u(0),c(0),z(0))+\int_{\Omega_t}W_{,e}^\mathrm{el}(e(u),c,z):e(\partial_t
    b)\,\mathrm dx\mathrm ds
    - \int_{\Omega_t}  \partial_t f \cdot u \,\mathrm dx\mathrm ds \\
    & \hspace{10.8cm}-\int_{\Omega} f(0) \cdot u(0) \,\mathrm dx
  \end{align*}
  and the variational inequality
  \begin{align}
    \label{eq:r}
    0\leq \left\langle \mathrm d_z\tilde{\mathcal E}(u(t),c(t),z(t))+r(t)+\mathrm d_{\dot z}
    \tilde{\mathcal R}(\partial_t z(t)),\zeta\right\rangle
  \end{align}
  for all $\zeta \in H_-^1(\Omega)\cap L^\infty(\Omega)$ and $r(t)\in \partial
  I(H_+^1(\Omega)\cap L^\infty(\Omega)
  ;z(t))$. 
\end{itemize}
Note that if one of the two properties are satisfied then we even obtain the
equation of balance of  energy:
\begin{align*}
  &\mathcal E(u(t),c(t),z(t))
  +\int_{0}^{t}\langle\mathrm d_{\dot z}\tilde{\mathcal R}(\partial_t z),\partial_t z\rangle\,\mathrm ds
  +\int_{\Omega_t} \mathcal \nabla w : {\mathbb M} \nabla w \,\mathrm dx
  \mathrm  ds - \int_{\Omega} f(t) \cdot u(t) \,\mathrm dx
  \notag\\
  & \quad=  \mathcal
  E(u(0),c(0),z(0))+\int_{\Omega_t}W_{,e}^\mathrm{el}(e(u),c,z):e(\partial_t b)\,\mathrm dx\mathrm ds
  - \int_{\Omega_t}  \partial_t f \cdot u \,\mathrm dx\mathrm ds - \int_{\Omega} f(0) \cdot u(0) \,\mathrm dx
\end{align*}
\end{proposition}
We would like to emphasize that the statement of Proposition
\ref{prop:energeticFormulation} is also true for the diffusion 
equation and the equation of balance of forces $(S_\eps)$ with initial-boundary
conditions \eqref{eqn:regUnifyingModelIBC} if we replace $ \mathcal E$ by $ \mathcal E_\eps$.

\begin{definition}[Weak solutions for the regularized system $\boldsymbol{ \eqref{eqn:regUnifyingModelClassical} }$]
  \label{def:weakSolutionRegularized}
  A quadruple $q_\eps=(u_\eps,c_\eps,$ $w_\eps,z_\eps)$ is called a weak solution of the regularized system
  \eqref{eqn:regUnifyingModelClassical} with the initial-boundary conditions
  \eqref{eqn:regUnifyingModelIBC} if the following properties are satisfied:
  \begin{enumerate}
    \renewcommand{\labelenumi}{(\roman{enumi})}
  \item
    Spaces \\
    The components of $q_\eps$ are in the following spaces:
    \begin{align*}
      &u_\eps \in  L^\infty(0,T;W^{1,4}(\Omega;\mathbb R^n)),\;u_\eps|_{D_T}=b|_{D_T},\\
      &c_\eps\in L^\infty(0,T;H^1(\Omega;\mathbb R^N))\cap
      H^1(0,T;(H^1(\Omega;\mathbb R^N))') ,\; c_\eps\in\Sigma
      \text{ a.e. in }\Omega_T,\\
      &z_\eps\in L^\infty(0,T;H_+^1(\Omega))\cap H^1(0,T;L^2(\Omega)), \,
      z_\eps(0)=z^0, \\[-2mm]
      \intertext{and} \\[-10mm]  
      &w_\eps\in L^2(0,T;H^1(\Omega;\mathbb R^N)).
    \end{align*}
    
  \item Diffusion \\
    For all $\zeta \in H^1(\Omega;\mathbb R^N)$ and for a.e. $t \in [0,T]$:
    \begin{align} 
      \label{eqn:regular1}
      &\int_{\Omega_T} \partial_t c_\eps (t)\cdot \zeta\,\mathrm dx \,  \mathrm dt
      =\int_{\Omega_T} \mathbb{M} \nabla w_\eps(t) :  \nabla \zeta \,  \mathrm dx
      \, \mathrm dt
    \end{align}
    For all $\zeta\in H^1(\Omega;\mathbb R^N)$ and for a.e. $t\in[0,T]$:	
    \begin{equation}
      \label{eqn:regular2}
      \begin{split}
	\int_{\Omega} w_\eps(t)\cdot\zeta\,\mathrm dx
	=\int_{\Omega}  \Big(\mathbb P\mathbf\Gamma\nabla c_\eps(t):\nabla\zeta
	+\mathbb P W_{,c}^\mathrm{ch}&(c_\eps(t))\cdot\zeta  \Big)\,\mathrm dx\\
	&+\int_{\Omega}\mathbb P
        W_{,c}^\mathrm{el}(e(u_\eps(t)),c_\eps(t),z_\eps(t))
        \cdot\zeta\,\mathrm dx 
      \end{split}
    \end{equation}
  \item Balance of forces\\
    For all $\zeta\in W_D^{1,4}(\Omega;\mathbb R^n)$ and for a.e. $t\in[0,T]$:
    \begin{equation}
      \label{eqn:regular3}
      \int_{\Omega} W_{,e}^\mathrm{el}(e(u_\eps(t)),c_\eps(t),z_\eps(t)):e(\zeta)\,\mathrm
      dx+ \eps \int_{\Omega}  |\nabla u_\eps(t)|^2 \nabla u_\eps(t) : \nabla \zeta  \,\mathrm dx
      = 	\int_{\Omega} f(t) \cdot \zeta  \,\mathrm dx
    \end{equation}
  \item Damage variational inequality\\
    For all $\zeta\in H_-^1(\Omega)$ and for a.e. $t\in[0,T]$:
    \begin{align}
      \label{eqn:regular4}
      0 & \le \int_{\Omega}  \big( \nabla z_\eps(t)\cdot\nabla\zeta
      +(W_{,z}^\mathrm{el}(e(u_\eps(t)),c_\eps(t),z_\eps(t))
      -\alpha+\beta(\partial_t z_\eps(t)))\zeta  \big)\,\mathrm dx ,\\
      0 & \le z_\eps(t),\\
      0 & \ge \partial_t z_\eps(t).
    \end{align}
  \item Energy inequality\\
    For a.e. $t\in[0,T]$:
    \begin{equation}
      \label{eqn:regular5}
      \begin{split}
        &\! \! \! \! \! \!  \mathcal E_\eps(u_\eps(t),c_\eps(t),z_\eps(t))
	+\int_\Omega \alpha (z^0-z_\eps(t))\,\mathrm dx+\int_{\Omega_t}
        \beta |\partial_t z_\eps|^2\,\mathrm dx\mathrm ds
	+ \int_{\Omega_t}    \nabla w_\eps :  \mathbb{M}  \nabla w_\eps
        \mathrm dx \mathrm ds \\ & \hspace{10.6cm}
        -\int_{\Omega} f(t) \cdot u_\eps(t) \,\mathrm dx\\
	& \leq \mathcal E_\eps(u_\eps^0,c^0,z^0)+
	\int_{\Omega_t}W_{,e}^\mathrm{el}(e(u_\eps),c_\eps,z_\eps):e(\partial_t b) \,\mathrm dx\mathrm ds+ \eps
        \int_{\Omega_t} |\nabla u_\eps|^2 \nabla u_\eps : \nabla \partial_t
        b \,\mathrm dx\mathrm ds\\ & \hspace{7.6cm}
        - \int_{\Omega_t}  \partial_t f \cdot u_\eps \,\mathrm dx\mathrm ds
        - \int_{\Omega} f(0) \cdot u_\eps(0) \,\mathrm dx,
      \end{split}
    \end{equation}
    where $u_\eps^0$ is the unique minimizer of $\mathcal
    E_\eps( \, \ldotp ,c^0,z^0) - \int_{\Omega} f(0) \cdot (\, \ldotp)\,  \mathrm dx  $ 
    in $W^{1,4}(\Omega;\mathbb R^n)$ with trace $u^0|_D=b(0)|_D$.
  \end{enumerate}
\end{definition}
Note that we can choose $r=0$ in \eqref{eq:r} due to $\Phi(0)=\Phi'(0)=0$, see
Lemma 3.7 and Remark 3.8 in \cite{HK10b} for details. 

\begin{definition}[Weak solution for the limit system 
$\boldsymbol{\eqref{eqn:unifyingModelClassical}}$]
  \label{def:weakSolutionLimit}
  A quadruple $q=(u,c,w,z)$ is called a weak solution of the system
  \eqref{eqn:unifyingModelClassical} with the initial-boundary conditions
  \eqref{eqn:unifyingModelIBC} if the following properties are satisfied:
  \begin{enumerate}
    \renewcommand{\labelenumi}{(\roman{enumi})}
    \item
      Spaces \\ 
      The components of $q$ are in the following spaces:
      \begin{align*}
	&u \in  L^\infty(0,T;W^{1,p}(\Omega;\mathbb R^n)),\;u|_{D_T}=b|_{D_T},\\
	&c\in L^\infty(0,T;H^1(\Omega;\mathbb R^N))\cap H^1(0,T;(H^1(\Omega;\mathbb R^N))') ,\; c\in\Sigma\text{ a.e. in }\Omega_T,\\
	&z\in L^\infty(0,T;H_+^1(\Omega))\cap
        H^1(0,T;L^2(\Omega)), \; z(0)=z^0, \\[-2mm]
        \intertext{and} \\[-10mm]  &w\in L^2(0,T;H^1(\Omega;\mathbb R^N)).
      \end{align*}
      
    \item Diffusion \\
      For all $\zeta\in L^2(0,T;H^1(\Omega;\mathbb R^N))$ with
      $\partial_t\zeta\in L^2(\Omega_T;\mathbb R^N)$ and $\zeta(T)=0$:
      \begin{align}
        \label{eqn:limit1}
	&\int_{\Omega_T}(c-c^0)\cdot\partial_t\zeta\,\mathrm dx \,  \mathrm dt
	=\int_{\Omega_T}  \mathbb{M} \nabla w :  \nabla \zeta \,  \mathrm dx
        \, \mathrm dt
      \end{align}
      For all $\zeta\in H^1(\Omega;\mathbb R^N)\cap L^\infty(\Omega;\mathbb
      R^N)$ and for a.e. $t\in[0,T]$:	
      \begin{equation}
        \label{eqn:limit2}
	\begin{split}
	  \int_{\Omega} w(t)\cdot\zeta\,\mathrm dx
	  =&\int_{\Omega}  \big( \mathbb P\mathbf\Gamma\nabla c(t):\nabla\zeta
	  +\mathbb P W_{,c}^\mathrm{ch}(c(t))\cdot\zeta   \big)\,\mathrm dx\\
	  &+\int_{\Omega}\mathbb P W_{,c}^\mathrm{el}(e(u(t)),c(t),z(t))\cdot\zeta\,\mathrm dx
	\end{split}
      \end{equation}
    \item Balance of forces\\
      For all $\zeta\in W_D^{1,p}(\Omega;\mathbb R^n)$ and for a.e. $t\in[0,T]$:
      \begin{equation}
         \label{eqn:limit3}
	\int_{\Omega} W_{,e}^\mathrm{el}(e(u(t)),c(t),z(t)):e(\zeta) \,\mathrm dx
        = \int_\Omega f(t)  \cdot \zeta  \,\mathrm dx
      \end{equation}
    \item Damage variational inequality\\
      For all $\zeta\in H_-^1(\Omega)\cap L^\infty(\Omega)$ and for a.e. $t\in[0,T]$:
      \begin{align}
        \label{eqn:limit4}
	0 & \le \int_{\Omega}  \big( \nabla z(t)\cdot\nabla\zeta
	+(W_{,z}^\mathrm{el}(e(u(t)),c(t),z(t))
	-\alpha+\beta(\partial_t z(t)))\zeta \big)\,\mathrm dx , \\
       0 & \le z(t),\label{eqn:limit4a} \\
       0 & \ge \partial_t z(t).\label{eqn:limit4b}
      \end{align}
    \item Energy inequality\\
      For a.e. $t\in[0,T]$:
      \begin{equation} \label{eqn:limit5}
	\begin{split}
	  &\mathcal E(u(t),c(t),z(t))
	  +\int_\Omega \alpha (z^0-z(t))\,\mathrm dx+\int_{\Omega_t} \beta |\partial_t z|^2\,\mathrm dx\mathrm ds
	  + \int_{\Omega_t}   \nabla w : \mathbb{M} \nabla w \,\mathrm
          dx\mathrm ds \\ & \hspace{9.8cm}-
          \int_{\Omega} f(t) \cdot u(t) \,\mathrm dx   \\
	  &\quad \leq\mathcal E(u^0,c^0,z^0)+
	  \int_{\Omega_t}W_{,e}^\mathrm{el}(e(u),c,z):e(\partial_t b)\,\mathrm
          dx\mathrm ds - \int_{\Omega_t}  \partial_t f \cdot u \,\mathrm dx\mathrm ds - \int_{\Omega} f(0) \cdot u(0) \,\mathrm dx,
	\end{split}
      \end{equation}
      where $u^0$ is the unique minimizer of $\mathcal 
      E( \, \ldotp ,c^0,z^0) - \int_{\Omega} f(0) \cdot (\, \ldotp)\,  \mathrm
      dx $ in $W^{1,p}(\Omega;\mathbb R^n)$ with trace $u^0|_D=b(0)|_D$.
  \end{enumerate}
\end{definition}

Note that both notions of weak solutions imply mass conservation, i.e.
\begin{align*}
  \int_\Omega c(t)\,\mathrm dx\equiv const.
\end{align*}

\subsection{Main results \label{sec:main}}
The main result of this work is the following theorem.
\begin{theorem}[Existence theorem]
  \label{theorem:mainTheorem}
  Let the assumptions of Section \ref{section:AssRes} be satisfied. Then for every
  \begin{align*}
    &b\in W^{1,1}(0,T;W^{1,\infty}(\Omega;\mathbb R^n)),\\
    &f\in W^{1,1}(0,T;L^{p'}(\Omega;\mathbb R^n))\text{ with }
    f^0=f(0)\in L^{p'}(\Omega;\mathbb R^n) , \\
    &c^0\in H^1(\Omega;\mathbb R^N)\text{ with }c^0\in\Sigma\text{ a.e. in }\Omega,\\
    &z^0\in H^1(\Omega)\text{ with }0\leq z^0\leq 1\text{ a.e. in }\Omega,
  \end{align*}
  there exists a weak solution $q$ of the system \eqref{eqn:unifyingModelClassical}
  in the sense of Definition \ref{def:weakSolutionLimit} 
  with the initial-boundary conditions \eqref{eqn:unifyingModelIBC}.
\end{theorem}

\section{Existence of weak solutions of $\boldsymbol{\eqref{eqn:unifyingModelClassical}}$ \label{sec:proof}}
By slight modifications of the proof of Theorem 2.5 in \cite{HK10b}, we can establish the following existence theorem.
 \begin{theorem}[Existence theorem, cf. \cite{HK10b}]
   \label{theorem:regularizedExistence} 
   Let the assumptions of Section \ref{section:AssRes} be satisfied. Then for every
   \begin{align*}
     &b\in W^{1,1}(0,T;W^{1,\infty}(\Omega;\mathbb R^n)),\\
     &f\in W^{1,1}(0,T;L^{p'}(\Omega;\mathbb R^n))\text{ with } f^0=f(0)\in L^{p'}(\Omega;\mathbb R^n) , \\
     &c^0\in H^1(\Omega;\mathbb R^N)\text{ with }c^0\in\Sigma\text{ a.e. in }\Omega,\\
     &z^0\in H^1(\Omega)\text{ with }0\leq z^0\leq 1\text{ a.e. in }\Omega,
   \end{align*}
   there exists a weak solution $q_\eps$ of the regularized system \eqref{eqn:regUnifyingModelClassical}
   in the sense of Definition \ref{def:weakSolutionRegularized} 
   with the initial-boundary conditions \eqref{eqn:regUnifyingModelIBC}.
 \end{theorem}
 
Next, we will show that an appropriate subsequence of the regularized
solutions $q_\varepsilon$ for $\varepsilon\in(0,1]$ of Definition
  \ref{def:weakSolutionRegularized} converges in ``some sense'' to $q$ which 
 satisfies the limit equations given in Definition
 \ref{def:weakSolutionLimit}. For each $\varepsilon\in(0,1]$, we denote with 
$q_\varepsilon=(u_\varepsilon,c_\varepsilon, w_\eps,z_\varepsilon)$ 
a solution according to Theorem \ref{theorem:regularizedExistence}.

\begin{lemma}
  \label{lemma:energyBoundedness}
  For a.e. $t\in [0,T]$, $t=0$ and every $\varepsilon\in(0,1]$: 
    \begin{equation}
      \begin{split}
        \label{eqn:energyBoundedness}
        \mathcal E_\varepsilon(u_\varepsilon(t),c_\varepsilon(t),z_\varepsilon(t))
        +\int_0^t\int_\Omega \big(-\alpha\partial_t z_\varepsilon
        +\beta|\partial_t z_\varepsilon|^2\big)\,\mathrm dx\mathrm ds
        & +\int_0^t \int_\Omega \nabla  w_\varepsilon : \mathbb{M}  \nabla  w_\varepsilon\,\mathrm dx\mathrm ds \\
        &  \leq C(\mathcal E_1(u_1^0,c^0,z^0)+1).
      \end{split}
    \end{equation}
\end{lemma}
\begin{proof}
In the following, $C>0$ denotes a context-dependent constant
independently of $t$ and $\varepsilon$. By means of  \eqref{eqn:W1growthEst4}
and \eqref{eqn:growthEst4}, we estimate for $ s \in[0,T]$:
\begin{align}
  \notag 
  & \int_\Omega \partial_e  W^\mathrm{el}(e(u_\eps(s),c_\eps(s),z_\eps(s)):e(\partial_t b(s))\, \mathrm dx\\
  \label{eqn:gronwallEstimate1}
  & \qquad\qquad\leq C\|\nabla\partial_t b(s)\|_{L^\infty(\Omega)}
  \int_\Omega \Big( W^\mathrm{el}(e(u_\eps(s)),c_\eps(s),z_\eps(s)) +1  \Big)\,\mathrm dx \\
  & \qquad \qquad\leq  C \|\nabla \partial_t b(s)\|_{L^\infty(\Omega)}  \big(\mathcal E_\varepsilon(e(u_\eps(s)), c_\eps(s), z_\eps(s)) +1 \big).
\end{align}
In addition, for $ s \in[0,T]$,
\begin{align}
  \eps \int_{\Omega}|\nabla u_\eps(s)|^2\nabla u_\eps(s) 
  & :\nabla\partial_t b(s)\,\mathrm dx \notag \\
  & \leq \eps  \|\nabla \partial_t b(s)\|_{L^\infty(\Omega)}
  \int_{\Omega}|\nabla u_\eps(s)|^3\,\mathrm dx \notag\\
  & \leq \eps C  \|\nabla \partial_t b(s)\|_{L^\infty(\Omega)}
  \Big( \int_{\Omega}|\nabla u_\eps(s)|^4\,\mathrm dx +1 \Big)  \notag\\
  &\leq  C \|\nabla \partial_t b(s)\|_{L^\infty(\Omega)}  \Big(\mathcal
  E_\varepsilon(e(u_\eps(s)), c_\eps(s), z_\eps(s)) +1 \Big)\,
  \label{eqn:gronwallEstimate2}
\end{align}
and 
\begin{align}
  \int_\Omega \partial_t f(s) \cdot u_\eps(s) \, \mathrm dx &\le C  
  \| \partial_t f(s)\|_{L^{p'}(\Omega)}   \| u_\eps(s)\|_{L^{p}(\Omega)}
  \notag\\ 
  & \le C \|  \partial_t f(s)\|_{L^{p'}(\Omega)}   \Big(\mathcal E_\varepsilon(e(u_\eps(s)), c_\eps(s), z_\eps(s)) +1 \Big),
  \label{eqn:gronwallEstimate3a}
\end{align}
\begin{align}
  \int_\Omega  f(0) \cdot u_\eps(0) \, \mathrm dx &\le 
  C  \| f(0)\|_{L^{p'}(\Omega)} \Big( \mathcal E_\varepsilon(e(u_\eps^0), c^0,
  z^0) +1 \Big),
  \label{eqn:gronwallEstimate3c}\\
  \int_\Omega  f(s) \cdot u_\eps(s) \, \mathrm dx 
  &\le  C  \|  f(s)\|^{p'}_{L^{p'}(\Omega)}   + \frac{1}{2} \Big(\mathcal E_\varepsilon(e(u_\eps(s)), c_\eps(s), z_\eps(s)) +1 \Big),
  \label{eqn:gronwallEstimate3b}
\end{align}
where the last inequality follows by the general Young's inequality. To simplify notation, we define the functions
\begin{multline*}
  \gamma_\eps(t):= \frac{1}{2} \mathcal E_\eps(e(u_\eps(t)), c_\eps(t), z_\eps(t))
  +  \int_\Omega \alpha (z^0{-}z_\eps(t))\,\mathrm dx+ 
  \int_{\Omega_t} \beta |\partial_t z_\eps|^2\,\mathrm dx\mathrm ds\\
  +  \int_{\Omega_t}   \nabla  w_\varepsilon : \mathbb{M}  \nabla  w_\varepsilon \,  \mathrm dx \mathrm ds 
\end{multline*}
and 
\begin{align*}
  h(s):=  \|\nabla \partial_t b(s)\|_{L^\infty(\Omega)}   +   \| \partial_t f(s)\|_{L^{p'}(\Omega)}.
\end{align*}
Using \eqref{eqn:gronwallEstimate1}--\eqref{eqn:gronwallEstimate3c}, the energy inequality
\eqref{eqn:regular5} of the regularized system can be estimated for a.e. $t
\in[0,T]$ as follows:
\begin{align*}
  \gamma_\eps  (t)\leq{} & \mathcal E_\varepsilon(e(u_\eps^0), c^0,z^0) + C
  +C\int_0^{t} h(s) \,
  \mathcal E_\varepsilon(e(u_\eps(s)), c_\eps(s), z_\eps(s))\,\mathrm ds   \\
  &\hspace{8.5cm}
  +C  \| f(0)\|_{L^{p'}(\Omega)} \mathcal E_\varepsilon(e(u_\eps^0), c^0, z^0) \\
  \leq{}&C\mathcal E_\varepsilon(e(u_\eps^0), c^0,z^0) + C
  +C\int_0^{t}  h(s)\,  \mathcal E_\varepsilon(e(u_\eps(s)), c_\eps(s), z_\eps(s))\,\mathrm ds\\
  \leq{}& C\mathcal E_\varepsilon(e(u_\eps^0), c^0,z^0) + C
  +C\int_0^{t} h(s) \,    \mathcal \gamma_\varepsilon(s)\,\mathrm ds\\
\end{align*}
Since $\mathcal E_\varepsilon(u_\eps^0,c^0,z^0) - \int_\Omega f(0)  \cdot u_\eps^0 \mathrm dx 
\leq\mathcal E_\varepsilon(u_1^0,c^0,z^0) - \int_\Omega f (0) \cdot  u_1^0 \mathrm dx \leq \mathcal E_1(u_1^0,c^0,z^0)
- \int_\Omega f (0) \cdot u_1^0 \mathrm dx  $ Gronwall's inequality shows for a.e. $t\in[0,T]$ and every $\eps \in (0,1]$:
\begin{align*}
  \gamma_\eps(t)&\leq C+ C\mathcal E_\varepsilon(e(u_\eps^0), c^0,z^0) \\
  & \hspace{2cm}
  +C\int_{0}^t (C+ C\mathcal E_\varepsilon(e(u_\eps^0), c^0,z^0)) h(s)
  \exp\left(\int_s^t  h(l) \,\mathrm dl\right)\,\mathrm ds\\
  &\leq C(\mathcal E_1(u_1^0,c^0,z^0)+1) . 
\end{align*}
\ep
\end{proof}
\begin{lemma}[A-priori estimates]
  \label{lemma:boundednessEpsilon}
  There exists some constant $C>0$ independently of $\varepsilon>0$ such that for all $\varepsilon\in(0,1]$:\\

\begin{tabular}[t]{ll}
  \begin{minipage}{20em}
    \begin{enumerate}
      \renewcommand{\labelenumi}{(\roman{enumi})}
    \item
      $\|u^0_\varepsilon\|_{ W^{1,p}(\Omega;\mathbb R^n)} \leq C$,\\[1mm]
      $\|u_\varepsilon\|_{L^\infty(0,T; W^{1,p}(\Omega;\mathbb R^n))}\leq C$,
    \item
      $\varepsilon^{1/4}\|u_\varepsilon\|_{L^\infty(0,T;W^{1,4}(\Omega;\mathbb R^n))}\leq C$,
    \item
      $\| \sqrt{\Phi(z^0)} \, e( u^0_\varepsilon) \|_{ L^{2}(\Omega;\mathbb R^{n\times n})}\leq C$, \\[1mm]                        
      $\| \sqrt{\Phi(z_\varepsilon)} \,  e( u_\varepsilon )\|_{L^\infty(0,T;
      L^{2}(\Omega;\mathbb
      R^{n\times n}))}\leq C$,                                                
    \end{enumerate}
  \end{minipage}
  &
  \begin{minipage}{22em}
    \begin{enumerate}
      \renewcommand{\labelenumi}{(\roman{enumi})}
    \item[(iv)]
      $\|c_\varepsilon\|_{L^\infty(0,T;H^1(\Omega;\R^N))}\leq C$,
    \item[(v)]
      $\| \partial_t
      c_\varepsilon\|_{L^2( 0,T;
        (H^{1}(\Omega; \R^N))')}\leq C$,
    \item[(vi)]
      $\|z_\varepsilon\|_{L^\infty(0,T;H^{1}(\Omega))}\leq C$,                                                          
    \item[(vii)]
      $\|\partial_t z_\varepsilon\|_{L^2(\Omega_T)}\leq C$,
    \item[(viii)]
      $\|w_\varepsilon\|_{L^2(0,T;H^1(\Omega;
      \R^N))}\leq
      C$.\\
    \end{enumerate}
  \end{minipage}
\end{tabular}
\end{lemma}

\begin{proof}
According to Lemma \ref{lemma:energyBoundedness}, 
we immediately obtain (vi) and (vii).  
Due to $\int_\Omega c_\varepsilon(t)\,\mathrm dx=\mathrm{const.}$ and the boundedness of
$\|\nabla c_\varepsilon(t)\|_{L^\infty(0,T;L^2(\Omega; \R^{N\times n}))}$, Poincar\'e's inequality yields (iv).
In addition, (ii) follows from Poincar\'e's inequality. 

From the growth conditions for $W_1^\mathrm{el}$ and $W_2^\mathrm{el}$, Lemma 
\ref{lemma:energyBoundedness}, Young's inequality, $1 < p <2$, 
$z \in [0,1]$ a.e.~in $\Omega_T$, $s < \frac{n(p-1)}{n-p}$ and (iv), we infer
for $t=0$, a.e.~$t\in[0,T]$ and any $\eps >0$:
\begin{equation}
  \begin{split}
    \int_\Omega & \big(\Phi ( z_\eps(t) )
    |e(u_\eps(t))|^2 +   |e(u_\eps(t))|^p \big) \, \mathrm dx\\
    & \le C_1 \int_\Omega\Big(  \Phi ( z_\eps(t) ) \big(
    W_1^\mathrm{el} (e(u_\eps(t)), c_\eps(t)) + | c_\eps(t)|^4 +1\big) 
    +  W_2^\mathrm{el} (e(u_\eps(t)), c_\eps(t)) + |
    c_\eps(t)|^{s p' } +1 \Big)\, \mathrm dx\\
    & \le C_1 \int_\Omega \big( W^\mathrm{el} (e(u_\eps(t)), c_\eps(t),
    z_\eps(t)) +  \Phi ( z_\eps(t) )
    W_2^\mathrm{el} (e(u_\eps(t)), c_\eps(t))  \big)\,
    \mathrm dx + C_2    \\
    & \le C_1 \int_\Omega W^\mathrm{el} (e(u_\eps(t)), c_\eps(t),
    z_\eps(t)    ) \, \mathrm dx+  C_3 \int_\Omega \Phi ( z_\eps(t) )   
    \big( |e(u_\eps(t))|^p   + |c_\eps(t)|^s +1   \big)\, \mathrm dx + C_2    \\
    &\le C_4 +  {C}_3 \int_\Omega \Phi (  z_\eps(t) )    |e(u_\eps(t))|^p   \, \mathrm dx  \\  
    & \le C_4 + {C}_3 \int_\Omega \big(\Phi (  z_\eps(t) ) \big)^{1-\frac{p}{2}}  \big(\Phi (
    z_\eps(t) )\big)^{\frac{p}{2}}   |e(u_\eps(t))|^p   \, \mathrm dx \\
    & \le C_4+  {C}_3 \,  \eps \int_\Omega \Phi (
    z_\eps(t) ) |e(u_\eps(t))|^2   \, \mathrm dx+ C(\eps)
  \end{split}
\end{equation}
Hence, we attain (i) by using the generalized Korn's inequality, see for
instance \cite{Nit81} and \cite{KO88}, and (iii). 

Due to \eqref{eqn:regular2} we obtain boundedness of
$\int_\Omega w_\varepsilon(t)\,\mathrm dx$. Since $\|\nabla w_\varepsilon(t)\|_{L^2(\Omega_T;
  \R^{N \times n})}$ is also bounded, Poincar\'e's inequality yields (viii).

Finally, we know from the boundedness of $\{\nabla w_\varepsilon\}$ in
$L^2(\Omega_T; \R^{N \times n})$ that $\{\partial_t c_\varepsilon\}$ is also 
bounded in $L^2( 0,T;(H^1(\Omega; \R^N))')$ with respect to $\varepsilon$ by equation
\eqref{eqn:regular1}. Therefore, (v) is satisfied. $\phantom{w}$
\ep
\end{proof}

\begin{lemma}[Convergence properties]
  \label{lemma:weakConvergenceEpsilon}
  There exists a subsequence $\{q_\epsk\}$ with $\epsk \searrow 0$  
  and a tuple $q=(u,c,w,z)$, satisfying (i) of Definition \eqref{def:weakSolutionLimit},
  $0\leq z\leq 1$ and $\partial_t z\leq 0$ a.e. in  $\Omega_T$, such that \\
  
  \begin{tabular}[t]{ll}
    \begin{minipage}{36em}
      \begin{enumerate}
        \renewcommand{\labelenumi}{(\roman{enumi})}
      \item
        $c_\epsk\stackrel{}{\rightharpoonup} c\text{ in }H^1(0,T;(H^{1}(\Omega; \R^N))')$,\\
        $c_\epsk(t)\rightharpoonup c(t)\text{ in }H^{1}(\Omega; \R^N)$ a.e. $t \in [0,T]$, \\
        $c_\epsk\rightarrow c\text{ a.e. in }\Omega_T$,\\
      \item
        $z_\epsk\stackrel{\star}{\rightharpoonup}
        z\text{ in }L^\infty(0,T; H^{1}(\Omega))$,\\
        $z_\epsk(t)\rightharpoonup
        z(t)\text{ in } H^{1}(\Omega)$ a.e. $t \in [0,T]$,\\
        $z_\epsk\rightarrow z\text{ a.e. in }\Omega_T$, \\
        $z_\epsk\rightharpoonup
        z\text{ in }  H^1(0,T;L^2(\Omega))$,\\
        $ z_\epsk \to z$ \text{ in } $L^{\hat{p}}(\Omega_T)$ for $\hat{p} \in [1,\infty)$,
      \end{enumerate}
      \begin{enumerate}
        \renewcommand{\labelenumi}{(\roman{enumi})}
      \item[(iii)]
        $u_\epsk\stackrel{\star}{\rightharpoonup}  u\text{ in} L^\infty(0,T;W^{1,p}(\Omega;\mathbb R^n))$,\\   
        $u^0_\epsk{\rightharpoonup} \, \, u^0\text{ in }W^{1,p}(\Omega;\mathbb R^n)$,\\
        $ \sqrt{\Phi (z_\epsk)}\,  e(  u_\epsk ) \stackrel{\star}{\rightharpoonup}
        \sqrt{\Phi ( z)} \, e( u)\text{ in }L^\infty(0,T;L^{2}(\Omega;\mathbb R^{n\times n}))$,\\
        $ \sqrt{\Phi (z^0)}\,e( u^0_\epsk ) \stackrel{}{\rightharpoonup} \sqrt{\Phi ( z^0)}\,
        e( u^0)\text{ in } L^{2}(\Omega;\mathbb R^{n\times n})$,
      \item[(iv)]
        $w_\epsk\rightharpoonup w\text{ in } L^2(0,T;H^1(\Omega;\R^N))$\\
      \end{enumerate}
    \end{minipage}
  \end{tabular}\\
  as $k \to \infty$.
\end{lemma}
\begin{proof}
  \begin{enumerate}
    \renewcommand{\labelenumi}{(\roman{enumi})}
  \item
    Properties (iv) and (v) of Lemma \ref{lemma:boundednessEpsilon}
    show that $\{ c_\epsk \}$ converges strongly to an element $c$ in $L^2(\Omega_T)$
    for a subsequence by a compactness result due to J.~P. Aubin and J.~L. Lions (see \cite{Simon}).
    This allows us to extract a further subsequence (still denoted by $\{\epsk\}$)
    such that $c_\epsk(t)\rightarrow c(t)$ in $L^2(\Omega; \R^N)$ for
    a.e. $t\in[0,T]$ as $k\to \infty$. 
    Taking also the boundedness of $\{c_\epsk\}$ in $L^\infty(0,T;H^1(\Omega;\R^N))$ into
    account, we obtain a subsequence with
    $c_\epsk(t)\rightharpoonup c(t)\text{ in }H^{1}(\Omega;\R^N)$ for
    a.e. $t\in[0,T]$. Moreover, $c_\epsk\rightarrow c$ a.e. in
    $\Omega_T$ with $c \in \Sigma$ as well as
    $c_\epsk\stackrel{\star}{\rightharpoonup}c$
    in $H^1(0,T;(H^1(\Omega;\R^N))')$ as $k\to \infty$.  
  \item
    These properties follow from the same argumentation as in (i) and the
    boundedness of $\{z_\epsk\}$ in $H^1(0,T;L^2(\Omega))$. 
    The function $z$ derived in this way is monotonically decreasing with
    respect to $t$, i.e. $\partial_t z\leq 0$ a.e. in $\Omega_T$ and $z \in
    [0,1]$ a.e.\,. By compact embeddings, we obtain the strong convergence results.
  \item
    Because of the boundedness of $\{u_\epsk\}$ and $\{u^0_\epsk\}$ in 
    $L^\infty(0,T;W^{1,p}(\Omega;\mathbb R^n))$ and $W^{1,p}(\Omega;\mathbb
    R^n)$, respectively, we obtain the first two properties. The other
    properties follow from the previous two, the boundedness of $\{ \sqrt{\Phi(z_\epsk)}
    e(u_\epsk)\}$ and $\{ \sqrt{\Phi(z^0)} e(u^0_\epsk)\}$ in
    $L^\infty(0,T;L^2(\Omega;\mathbb R^{n \times n})) $ and
    $L^2(\Omega;\mathbb R^{n \times n}), $ respectively, and (ii).
  \item
    This property follows from the boundedness of $\{w_\epsk\}$ in $L^2(0,T;H^1(\Omega;\R^N))$.
    \ep
  \end{enumerate}
\end{proof}

\begin{lemma}
  \label{lemma:weakConvergence3a}
  There exist sequences $\{q_\epsk\}$ and $\{q^0_\epsk\}$ with $\epsk \searrow 0$ such
  that the following properties are satisfied: 
  \begin{itemize}
  \item[(i)] There exist $\theta_{u;c;z} \in
    L^{\infty}(0,T;L^{2}(\Omega;\mathbb R^{n\times n}))$ 
    and $\theta^0_{u;c;z}  \in L^{2}(\Omega;\mathbb R^{n\times n} )$ with 
    \begin{equation*}
      \sqrt{\Phi(z_\epsk)} W^\mathrm{el}_{1,e} ( e(u_\epsk), c_\epsk)
      \stackrel{\star}{\rightharpoonup}    \theta_{u;c;z} \qquad \text{ in } 
      L^\infty(0,T;L^{2}(\Omega;\mathbb R^{n \times n}) ) 
    \end{equation*}
    and 
    \begin{equation*}
      \sqrt{\Phi(z^0)}W^\mathrm{el}_{1,e} ( e(u^0_\epsk), c^0)
      \stackrel{}{\rightharpoonup}    \theta^0_{u;c;z} \qquad \text{ in } 
      L^{2}(\Omega;\mathbb R^{n \times n}). \hspace{1.3cm} 
    \end{equation*}
    In particular,
    \begin{equation}
      \label{eqn:weakConvergenceW2lemma3a}
            {\Phi(z_\epsk)} W^\mathrm{el}_{1,e} ( e(u_\epsk), c_\epsk)
            \stackrel{\star}{\rightharpoonup}   \sqrt{\Phi(z)}  \, \theta_{u;c;z} \qquad \text{ in } 
            L^\infty(0,T;L^{2}(\Omega;\mathbb R^{n \times n}) ) 
    \end{equation}
    and 
    \begin{equation}
      {\Phi(z^0)} W^\mathrm{el}_{1,e} ( e(u^0_\epsk), c^0)
      \stackrel{}{\rightharpoonup}   \sqrt{\Phi(z^0)}  \, \theta^0_{u;c;z} \qquad \text{ in } 
      L^{2}(\Omega;\mathbb R^{n \times n})  .
    \end{equation}
  \item[(ii)]
    \bes
    \begin{split}
      \liminf_{k \to \infty} \iot 
      & \Phi(z_\epsk) W^\mathrm{el}_{1,e}(e(u_\epsk),c_\epsk) :  e(u_\epsk) \,\mathrm dx  \mathrm dt 
      \ge   \iot    \sqrt{\Phi(z)}  \, \theta_{u;c;z} :e(u) \,\mathrm dx  \mathrm dt
    \end{split}
    \ees
    and 
    \bes
    \begin{split}
      \liminf_{k \to \infty} 
      \int_\Omega &  \Phi(z^0) W^\mathrm{el}_{1,e}(e(u^0_\epsk),c^0) : 
      e(u^0_\epsk)  \,\mathrm dx  \ge 
      \int_\Omega    \sqrt{\Phi(z^0)}  \, \theta^0_{u;c;z} :e(u^0) \,\mathrm dx  .
    \end{split}
    \ees
  \end{itemize}
\end{lemma}
\begin{proof}
  To (i):  Since 
  \begin{equation*}
    \|  \sqrt{\Phi(z_\epsk)}   W^\mathrm{el}_{1,e} ( e(u_\epsk), c_\epsk) \|_{  L^\infty(0,T;L^{2}(\Omega;\mathbb R^{n\times n} )) }  \le C 
  \end{equation*}
  there exists some $\theta_{u;c;z} \in L^{\infty}(0,T;L^{2}(\Omega;\mathbb R^{n\times n} ))$ such that
  \begin{equation*}
    \sqrt{\Phi(z_\epsk)}  W^\mathrm{el}_{1,e} ( e(u_\epsk), c_\epsk)
    \stackrel{\star}{\rightharpoonup}    \theta_{u;c;z} \qquad \text{ in } 
    L^\infty(0,T;L^{2}(\Omega;\mathbb R^{n \times n}) ) .
  \end{equation*}
  In consequence, 
  \begin{equation*}
    \Phi(z_\epsk) \, W^\mathrm{el}_{1,e} ( e(u_\epsk), c_\epsk)
    \stackrel{\star}{\rightharpoonup}  \sqrt{\Phi(z)} \,   \theta_{u;c;z} \qquad
    \text{in } 
    L^\infty(0,T;L^{2}(\Omega;\mathbb R^{n \times n}) ) .
  \end{equation*}
  In the same way, we obtain the result for 
  $  \sqrt{\Phi(z^0)} W_{1,e}^\mathrm{el}  (e(u^0_\epsk),c^0) $.\\
  To (ii): Since 
  $ h_{\hat c} (\cdot) = W^\mathrm{el}_{1,e} (\cdot, \hat c) - W^\mathrm{el}_{1,e} (0, \hat c) $ 
  is one homogeneous we obtain by means of the uniform convexity assumption  
  \begin{equation*}
    \begin{split}
      \Big( 
      &\sqrt{\Phi(z)}  \big(  W^\mathrm{el}_{1,e} ( e(u), c_\epsk)  
      - W^\mathrm{el}_{1,e} ( 0, c_\epsk)  \big) 
      -  \sqrt{\Phi(z_\epsk)}  \big( W^\mathrm{el}_{1,e} ( e(u_\epsk), c_\epsk) 
      -W^\mathrm{el}_{1,e} ( 0, c_\epsk)  \big) \Big)\\
      & \hspace{9.3cm}:\big(\sqrt{ \Phi(z)}  \, e(u) -   \sqrt{\Phi(z_\epsk)} \,  e(u_\epsk)
      \big)\,  \\
      & = \Big(  \big(  W^\mathrm{el}_{1,e} (\sqrt{\Phi(z)}  e(u), c_\epsk)  
      - W^\mathrm{el}_{1,e} ( 0, c_\epsk)  \big) 
      -   \big( W^\mathrm{el}_{1,e} (   \sqrt{\Phi(z_\epsk)}   e(u_\epsk), c_\epsk) 
      -W^\mathrm{el}_{1,e} ( 0, c_\epsk)  \big) \Big) \\
      & \hspace{9.3cm}:\big(\sqrt{ \Phi(z)}  \, e(u) -   \sqrt{\Phi(z_\epsk)} \,  e(u_\epsk)
      \big)\,  \\
      &= C \big|  \sqrt{ \Phi(z)}  \, e(u) -   \sqrt{\Phi(z_\epsk)} \,  e(u_\epsk)  \big|^2 \ge 0 
    \end{split}
\end{equation*}
  Therefore, 
  \begin{equation*}
    \begin{split}
      &\liminf_{k \to \infty} \iot \sqrt{\Phi(z_\epsk)} \, W^\mathrm{el}_{1,e} ( e(u_\epsk), c_\epsk) :
      \sqrt{\Phi(z)} e(u) \,\mathrm dx  \mathrm dt
      \le  \liminf_{k \to \infty} \iot \Big( {\Phi(z)}  \, W^\mathrm{el}_{1,e} ( e(u), c_\epsk) : e(u) \\
      &\hspace{1.9cm}-  \sqrt{\Phi(z)}\, W^\mathrm{el}_{1,e} ( e(u), c_\epsk) :
      \sqrt{\Phi(z_\epsk)}\, e(u_\epsk) + {\Phi(z_\epsk)} \, W^\mathrm{el}_{1,e} ( e(u_\epsk), c_\epsk) :
      e(u_\epsk)  \Big) \,\mathrm dx  \mathrm dt
    \end{split}
  \end{equation*}
  since 
  \begin{equation*}
    \lim_{k \to \infty} \iot \Big( \sqrt{\Phi(z)}  \, W^\mathrm{el}_{1,e} ( 0, c_\epsk)  
    -  \sqrt{\Phi(z_\epsk)}  \, W^\mathrm{el}_{1,e} ( 0, c_\epsk)  \Big)
    :\big(\sqrt{ \Phi(z)}  \, e(u) -   \sqrt{\Phi(z_\epsk)} \,  e(u_\epsk)
    \big)\,   \mathrm dx  \mathrm dt=0.
  \end{equation*}
  By (i), Lemma \ref{lemma:weakConvergenceEpsilon}, the growth 
  assumptions on  $ W_{1,e}^\mathrm{el}$ and 
  the generalized Lebesgue's convergence theorem,  
  we can pass to the limit:
  \begin{equation*}
    \begin{split}
      \iot \sqrt{\Phi(z)} \, \theta_{u;c;z} :
      e(u) \,\mathrm dx  \mathrm dt
      \le 
      \liminf_{k \to \infty} \iot 
             {\Phi(z_\epsk)} \, W^\mathrm{el}_{1,e} ( e(u_\epsk), c_\epsk) :
             e(u_\epsk)  \,\mathrm dx  \mathrm dt
    \end{split}
  \end{equation*}
  Analogously, we obtain the second assertion of (ii).
  \ep
\end{proof}
\begin{lemma}
\label{lemma:weakConverge3}
There exist sequences $\{q_\epsk\}$ and $\{q^0_\epsk\}$ with $\epsk \searrow
0$ such that the following properties are satisfied: 
\begin{itemize}
\item[(i)] There exist an $\eta_{u;c} \in L^{\infty}(0,T;L^{p'}(\Omega;\mathbb R^{n\times n}
  ))$ and $\eta^0_{u;c}  \in L^{p'}(\Omega;\mathbb R^{n\times n} )$ with 
  \begin{equation*}
    W^\mathrm{el}_{2,e} ( e(u_\epsk), c_\epsk)
    \stackrel{\star}{\rightharpoonup}    \eta_{u;c} \qquad \text{ in } 
    L^\infty(0,T;L^{p'}(\Omega;\mathbb R^{n \times n}) ) 
  \end{equation*}
and 
\begin{equation*}
  W^\mathrm{el}_{2,e} ( e(u^0_\epsk), c^0)
  \stackrel{}{\rightharpoonup}    \eta^0_{u;c} \qquad \text{ in } 
  L^{p'}(\Omega;\mathbb R^{n \times n}). \hspace{1.3cm} 
\end{equation*}
In particular,
\begin{equation}
  \label{eqn:weakConvergenceW2lemma}
  (1-\Phi(z_\epsk)) W^\mathrm{el}_{2,e} ( e(u_\epsk), c_\epsk)
  \stackrel{\star}{\rightharpoonup}   (1-\Phi(z)) \, \eta_{u;c} \qquad \text{ in } 
  L^\infty(0,T;L^{p'}(\Omega;\mathbb R^{n \times n}) ) 
\end{equation}
and 
\begin{equation}
  (1-\Phi(z^0)) W^\mathrm{el}_{2,e} ( e(u^0_\epsk), c^0)
  \stackrel{}{\rightharpoonup}   (1-\Phi(z^0)) \, \eta_{u;c} \qquad \text{ in } 
  L^{p'}(\Omega;\mathbb R^{n \times n})  .
\end{equation}
Furthermore,
\bes
\begin{split}
  & \liminf_{k \to \infty} 
  \iot  (1-\Phi(z_\epsk)) W^\mathrm{el}_{2,e}(e(u_\epsk),c_\epsk) : 
  e(u_\epsk)  \,\mathrm dx  \mathrm dt \ge 
  \iot   (1- {\Phi(z)})  \, \eta_{u;c} :e(u) \,\mathrm dx  \mathrm dt
\end{split}
\ees
and 
\bes
\begin{split}
  &\liminf_{k \to \infty} 
  \int_\Omega  (1- \Phi(z^0) ) W^\mathrm{el}_{2,e}(e(u^0_\epsk),c^0) : 
  e(u^0_\epsk)  \,\mathrm dx  \ge 
  \int_\Omega   (1- {\Phi(z^0)})  \, \eta^0_{u;c} :e(u^0) \,\mathrm dx  .
\end{split}
\ees
\item[(ii)]
  For any $\zeta \in L^1(0,T; W_D^{1,4}(\Omega; \R^{n }) )$:
  \bes
  \begin{split}
    \lim_{k \to \infty} \iot 
    & \Big( \Phi(z_\epsk) W^\mathrm{el}_{1,e}(e(u_\epsk),c_\epsk) :  \nabla \zeta  \\
    & \hspace{1cm} + (1- \Phi(z_\epsk))  \,W^\mathrm{el}_{2,e}(e(u_\epsk),c_\epsk) : 
    \nabla \zeta +  \epsk   \left| \nabla u_\epsk \right|^2
    \nabla u_\epsk: \nabla \zeta  \Big) \,\mathrm dx  \mathrm dt \\
    &= \iot   \Big( \sqrt{\Phi(z)} \,  \theta_{u;c;z} : 
    \nabla \zeta  + (1- \Phi(z)) \, \eta_{u;c}:
    \nabla \zeta \Big) \,\mathrm dx  \mathrm dt
  \end{split}
  \ees
  For any $\zeta  \in W_D^{1,4}(\Omega; \R^{n } )$:
  \bes
  \begin{split}
    \lim_{k \to \infty} 
    \int_\Omega & \Big( \Phi(z^0) W^\mathrm{el}_{1,e}(e(u^0_\epsk),c^0) : 
    \nabla \zeta  \\
    & \hspace{1cm}
    + (1- \Phi(z^0)) \,
    W^\mathrm{el}_{2,e}(e(u^0_\epsk),c^0) : \nabla \zeta +  \epsk   \left| \nabla u^0_\epsk \right|^2
    \nabla u^0_\epsk: \nabla \zeta  \Big) \,\mathrm dx \\
    &= \int_\Omega   \Big(  \sqrt{\Phi(z^0)} \, \theta^0_{u;c;z} : 
    \nabla \zeta  + (1- \Phi(z^0)) \, \eta^0_{u;c}:
    \nabla \zeta \Big) \,\mathrm dx  
  \end{split}
  \ees
\item[(iii)] For any $\zeta \in L^1(0,T; W_D^{1,p}(\Omega; \R^{n } ))$:
\bes
\begin{split}
  \iot  \Big( \sqrt{\Phi(z)} \, \theta_{u;c;z}:
  \nabla \zeta  + (1- \Phi(z)) \, \eta_{u;c}:
  \nabla \zeta \Big) \,\mathrm dx  \mathrm dt= \iot f \cdot  \zeta \,\mathrm dx  \mathrm dt
\end{split}
\ees
For any $\zeta \in W_D^{1,p}(\Omega; \R^{n } )$:
\bes
\begin{split}
\int_\Omega \Big( \sqrt{\Phi(z)} \, \theta^0_{u;c;z}: \nabla \zeta  + (1-
\Phi(z^0)) \, \eta^0_{u;c}: \nabla \zeta \Big) \,\mathrm dx 
= \int_\Omega f^0 \cdot  \zeta \,\mathrm dx  
\end{split}
\ees
\end{itemize}
\end{lemma}
\begin{proof}
To (i): 
Since 
\begin{equation*}
  \|  W^\mathrm{el}_{2,e} ( e(u_\epsk), c_\epsk) \|_{  L^\infty(0,T;L^{p'}(\Omega;\mathbb R^{n\times n} )) }  \le C 
\end{equation*}
there exists an $\eta_{u;c} \in L^{\infty}(0,T;L^{p'}(\Omega;\mathbb R^{n\times n}))$ such that
\begin{equation*}
  W^\mathrm{el}_{2,e} ( e(u_\epsk), c_\epsk)
  \stackrel{\star}{\rightharpoonup}    \eta_{u;c} \qquad \text{ in } 
  L^\infty(0,T;L^{p'}(\Omega;\mathbb R^{n \times n}) ) .
\end{equation*}
In consequence, 
\begin{equation*}
  (1-\Phi(z_\epsk)) \, W^\mathrm{el}_{2,e} ( e(u_\epsk), c_\epsk)
  \stackrel{\star}{\rightharpoonup} (1-\Phi(z)) \,   \eta_{u;c} \qquad
  \text{in} L^\infty(0,T;L^{p'}(\Omega;\mathbb R^{n \times n}) ) .
\end{equation*}
In the same way, we obtain the result for $ W_{2,e}^\mathrm{el}(e(u^0_\epsk),c^0) $.\\
The convexity condition for  $W^\mathrm{el}_{2}$ implies 
\begin{multline*}
  \iot  (1-\Phi(z_\epsk)) W^\mathrm{el}_{2}(e(u_\epsk),c_\epsk)   
  \,\mathrm dx  \mathrm dt \ge 
  \iot   (1- {\Phi(z_\epsk)}) W^\mathrm{el}_{2}(e(u),c_\epsk)  
  \,\mathrm dx  \mathrm dt \\ + 
  \iot   (1- {\Phi(z_\epsk)}) W^\mathrm{el}_{2,e}(e(u),c_\epsk): (e(u_\epsk)-e(u))  
  \,\mathrm dx  \mathrm dt .
\end{multline*}
Since,  for a suitable sequence,
$(1- {\Phi(z_\epsk)}) W^\mathrm{el}_{2}(e(u),c_\epsk) \to (1- {\Phi(z)}) W^\mathrm{el}_{2}(e(u),c) $ strongly in 
$L^{1}(\Omega_T)$ and  $  (1- {\Phi(z_\epsk)}) W^\mathrm{el}_{2,e}(e(u),c_k)
\to (1- {\Phi(z)}) W^\mathrm{el}_{2,e}(e(u),c) $ strongly in
$L^{p'}(\Omega_T)$ by Lebesgue's generalized convergence theorem, and  
$e(u_\epsk) \stackrel{}{\rightharpoonup} e(u)$ in $L^{p}(\Omega_T)$ we obtain 
\begin{equation}
  \label{eq:convex_W2}
  \liminf_{k \to \infty} 
  \iot  (1-\Phi(z_\epsk)) W^\mathrm{el}_{2}(e(u_\epsk),c_\epsk)   
  \,\mathrm dx  \mathrm dt \ge 
  \iot   (1- {\Phi(z)}) W^\mathrm{el}_{2}(e(u),c) \,\mathrm dx  \mathrm dt.
\end{equation}
From the convexity condition for $ W^\mathrm{el}_{2} $ we further deduce 
\begin{multline}
  \label{eq:1convex_W2_e}
  \iot  (1-\Phi(z_\epsk)) W^\mathrm{el}_{2}(e(u),c_\epsk)   
  \,\mathrm dx  \mathrm dt \ge 
  \iot   (1- {\Phi(z_\epsk)}) W^\mathrm{el}_{2}(e(u_\epsk),c_\epsk)  
  \,\mathrm dx  \mathrm dt \\  
  + \iot   (1- {\Phi(z_\epsk)}) W^\mathrm{el}_{2,e}(e(u_\epsk),c_\epsk): (e(u)-e(u_\epsk))  
  \,\mathrm dx  \mathrm dt .
\end{multline}
Equation \eqref{eq:1convex_W2_e} may be rewritten as 
\begin{multline*}
  \iot  (1-\Phi(z_\epsk)) W^\mathrm{el}_{2,e}(e(u_\epsk),c_\epsk):e(u_\epsk)   
  \,\mathrm dx  \mathrm dt \ge 
  \iot   (1- {\Phi(z_\epsk)}) W^\mathrm{el}_{2}(e(u_\epsk),c_\epsk)  
  \,\mathrm dx  \mathrm dt   \\ 
  -  \iot (1- {\Phi(z_\epsk)}) W^\mathrm{el}_{2}(e(u),c_\epsk)  \,\mathrm dx  \mathrm dt  +
  \iot   (1- {\Phi(z_\epsk)}) W^\mathrm{el}_{2,e}(e(u_\epsk),c_\epsk) : e(u)   \,\mathrm dx  \mathrm dt .
\end{multline*}
Applying the $\liminf$ on both sides and taking \eqref{eq:convex_W2} and
\eqref{eqn:weakConvergenceW2lemma} into account gives 
\begin{equation}
  \begin{split}
    \liminf_{k \to \infty} 
    \iot  (1-\Phi(z_\epsk)) 
    & W^\mathrm{el}_{2,e}(e(u_\epsk),c_\epsk):e(u_\epsk)   
    \,\mathrm dx  \mathrm dt  \\\ge 
    &\liminf_{k \to \infty} 
    \iot   (1- {\Phi(z_\epsk)}) W^\mathrm{el}_{2}(e(u_\epsk),c_\epsk)  
    \,\mathrm dx  \mathrm dt   \\ 
    & -  \iot (1- {\Phi(z)}) W^\mathrm{el}_{2}(e(u),c)  \,\mathrm dx  \mathrm dt  +
    \iot   (1- {\Phi(z)}) \eta_{u;c} :e(u)  
    \,\mathrm dx  \mathrm dt  \\
    \ge & \iot   (1- {\Phi(z)}) \eta_{u;c} :e(u)  
    \,\mathrm dx  \mathrm dt . 
  \end{split}
\end{equation}
By similar arguments, we derive the claim with the initial data.\\
To (ii): Let $\zeta\in L^1(0,T; W_D^{1,4}(\Omega; \R^{n }) )$ be arbitrary. 
By Lemma \ref{lemma:weakConvergenceEpsilon},  Lemma
\ref{lemma:weakConvergence3a} and 
(i), we can pass to the limit in equation \eqref{eqn:regular3}. More
precisely, we obtain 
\begin{align}
  0& =\lim_{k \to \infty} \bigg( \int_{\Omega_{T}} \Big(
  \Phi(z_\epsk)  W_{1,e}^\mathrm{el}  (e(u_\epsk),c_\epsk):e(\zeta) 
  + (1 - \Phi(z_\epsk) ) \, W_{2,e}^\mathrm{el}(e(u_\epsk),c_\epsk):e(\zeta) \Big)
  \,\mathrm
  dx  \mathrm dt  \notag \\
  & \hspace{5.4cm}+  \epsk \iot   \left| \nabla u_\epsk \right|^2
  \nabla u_\epsk: \nabla \zeta  \, \mathrm dx  \mathrm dt
  \bigg) -  \iot  f \cdot \zeta  \, \mathrm dx   \mathrm dt \notag \\
  & =    \iot \Big( \sqrt{\Phi(z)} \, \theta_{u;c;z}: \nabla \zeta + 
  (1 - \Phi(z) ) \, \eta_{u;c} : \nabla \zeta \Big) \,\mathrm dx   \mathrm dt
  -\iot  f \cdot \zeta  \, \mathrm dx   \mathrm dt \label{eqn:dens}
\end{align} 
by noticing
\begin{align*}
  \left|\iot   \epsk|\nabla u_\eps|^2 \nabla u_\epsk:\nabla \zeta\,\mathrm dx \mathrm dt\right|
  \leq \epsk \| u_\epsk\|_{L^\infty(0,T; W^{1,4}(\Omega; \R^{n }) )}^3
  \|\zeta\|_{ L^1(0,T; W^{1,4}(\Omega; \R^{n }) )}\rightarrow 0.
\end{align*}
Now let $\zeta\in  W_D^{1,4}(\Omega; \R^{n })$ be arbitrary. 
By Lemma \ref{lemma:weakConvergenceEpsilon}, Lemma
\ref{lemma:weakConvergence3a}, (i) and the 
fact that $u^0_\eps$ is a minimizer of 
$ \mathcal E_\eps (\, \ldotp, c^0,z^0) - \int_\Omega f \cdot ( \, \ldotp) \, \mathrm dx $ we deduce
\begin{align*}
  \begin{split}
    0& =\lim_{k \to \infty} \bigg( \int_{\Omega} \Big(
    \Phi(z^0)  W_{1,e}^\mathrm{el}  (e(u^0_\epsk),c^0):e(\zeta) 
    + (1 - \Phi(z^0) ) \, W_{2,e}^\mathrm{el}(e(u^0_\epsk),c^0):e(\zeta)  \Big)\,\mathrm
    dx  \\
    & \hspace{5.7cm}+  \epsk \int_\Omega   \left| \nabla u^0_\epsk \right|^2
    \nabla u^0_\epsk: \nabla \zeta  \, \mathrm dx  
    \bigg)  -  \int_\Omega f^0 \cdot \zeta  \, \mathrm dx  \\
    & =    \int_\Omega  \Big( \sqrt{\Phi(z^0)} \, \theta^0_{u;c;z} :  \nabla \zeta  + 
    (1 - \Phi(z^0) ) \, \eta^0_{u;c} : \nabla \zeta  \Big)\,\mathrm dx  
    -\int_\Omega  f^0 \cdot \zeta  \, \mathrm dx \, .
  \end{split}
\end{align*}
To (iii):  
Since $f \in L^\infty(0,T; L^{p'}(\Omega; \R^{n } ))$ and $(1- \Phi(z)) \,  \eta_{u;c} \in 
L^\infty(0,T;L^{p'}(\Omega; \R^{n \times n } ))$ we obtain from
\eqref{eqn:dens} the claim by a density argument. Analogously, we 
derive the second claim for  $q^0_\epsk= (u^0_\epsk,c^0,z^0)$. $\phantom{w}$
\ep
\end{proof}

\begin{lemma}
  \label{lemma:weakConverge4}
  There exist  sequences $\{ q_\epsk\}$ and  $\{ q^0_\epsk \}$  with 
  $ \epsk \searrow 0$ such that 
  \begin{itemize}
  \item[(i)] 
    \begin{samepage}
      \begin{align}
        \lim_{k \to \infty} &\iot {\Phi(z_\epsk)} \, W^\mathrm{el}_{1,e}(e(u_\epsk),c_\epsk)
        :e(u _\epsk) \,\mathrm dx  \mathrm dt= \iot 
        \sqrt{\Phi(z)}  \, \theta_{u;c;z}:  e( u) \,\mathrm dx  \mathrm dt, \label{eqn:auxPoint1} \\
        \lim_{k \to \infty} &\iot (1 -\Phi(z_\epsk)) \, W^\mathrm{el}_{2,e}(e(u_\epsk),c_\epsk)
        :e(u_\epsk)\,\mathrm dx  \mathrm dt = \notag\\ 
        & \hspace{6cm}\iot  (1- \Phi(z)) \, \eta_{u;c}: e( u ) \,\mathrm dx  \mathrm dt\label{eqn:auxPoint2}, \\
        \lim_{k \to \infty} & \iot \epsk | \nabla u_\epsk |^4\,\mathrm dx
        \mathrm dt=0,
      \end{align}
    \end{samepage}
  \item[(ii)]
    \begin{align*}
      \lim_{k \to \infty} 
      &\int_\Omega {\Phi(z^0)} \, W^\mathrm{el}_{1,e}(e(u^0_\epsk),c^0)
      :e(u^0_\epsk) \,\mathrm dx  =
      \int_\Omega \sqrt{\Phi(z^0)}  \ \, \theta^0_{u;c;z}:
      \nabla u^0 \,\mathrm dx ,  \\
      \lim_{k \to \infty} & \int_\Omega (1 -\Phi(z^0)) \,
      W^\mathrm{el}_{2,e}(e(u^0_\epsk),c^0) :e(u^0_\epsk)\,\mathrm dx = \notag
      \int_\Omega (1- \Phi(z^0)) \, \eta^0_{u;c} \,\mathrm dx ,\\
      \lim_{k \to \infty} & \int_\Omega  \epsk | \nabla u^0_\epsk |^4\,\mathrm dx  =0.
    \end{align*}
  \end{itemize}
\end{lemma}
\begin{proof}
We obtain by \eqref{eqn:regular3} and Lemma \ref{lemma:weakConverge3} (iii)
\begin{equation*}
  \begin{split}
    \lim_{k \to \infty} & \bigg( \int_{\Omega_{T}}\Big(
    \Phi(z_\epsk)  W_{1,e}^\mathrm{el}  (e(u_\epsk),c_\epsk):e(u_\epsk -b)  \\
    & + (1 - \Phi(z_\epsk) ) \, W_{2,e}^\mathrm{el}(e(u_\epsk),c_\epsk):e(u_\epsk-b) \Big)\,\mathrm
    dx  \mathrm dt  +   \iot   \epsk \left| \nabla u_\epsk \right|^2
    \nabla u_\epsk: \nabla (u_\epsk -b)  \, \mathrm dx  \mathrm dt\bigg) \\
    & =\lim_{k \to \infty}  \iot  f \cdot ( u_\epsk -b ) \, \mathrm dx   \mathrm dt
    =   \iot  f \cdot ( u -b ) \, \mathrm dx   \mathrm dt\\
    &   =  \iot  \sqrt{\Phi(z) } \theta_{u;c;z}: \nabla (u -b)   +(1 - \Phi(z) )  \,
    \eta_{u;c}   : \nabla (u -b)   \,\mathrm dx   \mathrm dt \\
    &   =  \iot  \sqrt{\Phi(z) } \theta_{u;c;z}: e (u -b)  
    +(1 - \Phi(z) )  \, \eta_{u;c}  : e (u -b)   \,\mathrm dx   \mathrm dt. 
  \end{split}
\end{equation*}
Because of Lemma \ref{lemma:boundednessEpsilon} (ii), Lemma \ref{lemma:weakConvergenceEpsilon}, 
Lemma \ref{lemma:weakConvergence3a} and Lemma \ref{lemma:weakConverge3} 
\begin{equation*}
  \begin{split}
    &\lim_{k \to \infty}  \int_{\Omega_{T}}
    \Phi(z_\epsk)  W_{1,e}^\mathrm{el}  (e(u_\epsk),c_\epsk):e(b)  \, \mathrm dx  \mathrm dt
    = \int_{\Omega_{T}} \sqrt{\Phi(z) }\, \theta_{u;c;z}:e(b)  \, \mathrm dx  \mathrm dt,\\
    & \lim_{k \to \infty}\int_{\Omega_{T}} (1 - \Phi(z_\epsk) ) \, 
    W_{2,e}^\mathrm{el}(e(u_\epsk),c_\epsk):e(b)\,\mathrm dx  \mathrm dt  
    = \int_{\Omega_{T}} (1 - \Phi(z) ) \, 
    \eta_{u;c}: \nabla b \,\mathrm dx  \mathrm dt ,\\
    & \lim_{k \to \infty}  \iot   \epsk \left| \nabla u_\epsk \right|^2
    \nabla u_\epsk: \nabla b  \, \mathrm dx  \mathrm dt =0.
  \end{split}
\end{equation*}
Hence,
\begin{equation} \label{eqn:lower}
  \begin{split}
    \lim_{k \to \infty} & \Big( \int_{\Omega_{T}} \Big(
    \Phi(z_\epsk) W_{1,e}^\mathrm{el}(e(u_\epsk),c_\epsk):e(u_\epsk) 
    + (1 - \Phi(z_\epsk) ) \, W_{2,e}^\mathrm{el}(e(u_\epsk),c_\eps):e(u_\epsk)
    \Big)\,\mathrm
    dx  \mathrm dt   \\ 
    & \hspace{9.0686cm} + \iot   \epsk \left| \nabla u_\epsk \right|^4 \,\mathrm dx  \mathrm dt \Big) \\
    &   =  \iot  \Big( \sqrt{\Phi(z) } \, \theta_{u;c;z}:e(u )  
    +(1 - \Phi(z) ) \,  \eta_{u;c}: \nabla u  \Big) \,\mathrm dx   \mathrm dt.
  \end{split}
\end{equation}
The lower semicontinuity of all three terms on the left hand side in
\eqref{eqn:lower} implies the claim. \\
The assertion for $\{ q_\epsk^0 \} $ can be derived by slight modifications. 
\ep
\end{proof}

\begin{lemma} \label{corollary:pointwise}
  There exist  subsequences $\{q_\epsk\}$ and $\{q^0_\epsk\}$ with 
  $\epsk \searrow 0$ such that 
  \begin{itemize}
  \item[(i)]
    \begin{align}
      \label{eqn:con_p_1}
      \sqrt{\Phi(z_\epsk)}  e( u_\epsk)  &\to  \sqrt{\Phi(z)}  e(u)  \hspace{2.1686cm} 
      \text{ in } L^2(\Omega_T; \R^{n\times n}) ,\\[1mm] 
      \label{eqn:con_p_2}
      (1-\Phi(z_\epsk))^{1/p}  e( u_\epsk ) &\to (1- \Phi(z))^{1/p}  e( u )\qquad
      \hspace{0.5cm}  \text{ in } L^p(\Omega_T; \R^{n\times n})  , \\[1mm] 
      \nabla u_\epsk  & \to   \nabla u  \hspace{3.25cm} \ \text{ in } 
      L^p(\Omega_T; \R^{n\times n)} , \notag\\[1mm] 
      \nabla u_\epsk  & \to \nabla u  \qquad  
      \hspace{2.8cm} \text{a.e.} \text{ in }  \Omega_T \notag.
    \end{align}
  \item[(ii)]
    \begin{equation*}
      \begin{split}
        \sqrt{\Phi(z^0)} e( u^0_\epsk ) &\to  \sqrt{\Phi(z^0)} e( u^0 )\qquad
        \hspace{1.28cm}  \text{ in } 
        L^2(\Omega; \R^{n\times n}) ,\\[1mm] 
        (1-\Phi(z^0))^{1/p} e(u^0_\epsk)  &\to (1- \Phi(z^0))^{1/p}  e (u^0)
        \qquad
        \hspace{2mm} \text{ in } L^p(\Omega; \R^{n\times n})  , \\[1mm] 
        \nabla u^0_\epsk  &\to   \nabla u^0 \hspace{3.15cm}
        \hspace{1mm} \text{ in } L^p(\Omega; \R^{n\times n})  , \\[1mm] 
        \nabla u^0_\epsk & \to \nabla u^0 \qquad  \hspace{2.62cm} \text{a.e.} \text{ in }  \Omega_T.\\[1mm]
      \end{split}
    \end{equation*}
  \end{itemize}
\end{lemma}
\begin{proof}
Because of the uniform convexity condition for $ W^\mathrm{el}_{1}$ and the one homogeneity of 
$ h_{\hat c}(\cdot)= W^\mathrm{el}_{1,e} (\cdot, \hat c) - W^\mathrm{el}_{1,e} (0, \hat c) $ we get
\begin{equation*}
  \begin{split}
    &\limsup_{k \to \infty} \iot C |\sqrt{ \Phi(z) }e(u) -  \sqrt{ \Phi(z_\epsk) }\,
    e(u_\epsk)|^2 \,\mathrm dx  \mathrm dt \le  \limsup_{k \to \infty} \iot
    \Big( \sqrt{\Phi(z)}  \, W^\mathrm{el}_{1,e} ( e(u), c_\epsk)  \\
    &\hspace{2.8cm}-  \sqrt{\Phi(z_\epsk)}  \, W^\mathrm{el}_{1,e} ( e(u_\epsk), c_\epsk)  \Big) 
    :\big(\sqrt{ \Phi(z)}  \, e(u) -   \sqrt{\Phi(z_\epsk)} \,  e(u_\epsk) \big)    \,\mathrm dx  \mathrm dt 
  \end{split}
\end{equation*}
as 
\begin{equation*}
  \lim_{k \to \infty} \iot \Big( \sqrt{\Phi(z)}  \, W^\mathrm{el}_{1,e} ( 0, c_\epsk)  
  -  \sqrt{\Phi(z_\epsk)}  \, W^\mathrm{el}_{1,e} ( 0, c_\epsk)  \Big)
  :\big(\sqrt{ \Phi(z)}  \, e(u) -   \sqrt{\Phi(z_\epsk)} \,  e(u_\epsk)
  \big)\,   \mathrm dx  \mathrm dt=0.
\end{equation*}
Since, for a suitable sequence (also denoted by $\{ \eps_k \}$), 
\begin{equation*}
  \begin{split}
    &\lim_{k \to \infty} \iot \sqrt{\Phi(z_\epsk)} \, W^\mathrm{el}_{1,e} ( e(u_\epsk), c_\epsk) :
    \sqrt{\Phi(z)} e(u) \,\mathrm dx  \mathrm dt
    =\iot \sqrt{\Phi(z)} \, \theta_{u;c;z} : e(u) \,\mathrm dx  \mathrm dt,  \\
    &  \lim_{k\to \infty} \iot {\Phi(z)}  \, W^\mathrm{el}_{1,e} ( e(u), c_\epsk) : e(u) 
    \,\mathrm dx  \mathrm dt = \iot {\Phi(z)}  \, W^\mathrm{el}_{1,e} ( e(u),
    c) : e(u) \,\mathrm dx  \mathrm dt ,\\
    &  \lim_{k\to \infty} \iot   \sqrt{\Phi(z)}\, W^\mathrm{el}_{1,e} ( e(u), c_\epsk) :
    \sqrt{\Phi(z_\epsk)}\, e(u_\epsk)\,\mathrm dx  \mathrm dt
    = \iot   \sqrt{\Phi(z)}\, W^\mathrm{el}_{1,e} ( e(u), c) :
    \sqrt{\Phi(z)}\, e(u)\,\mathrm dx  \mathrm dt,\\ 
    &   \lim_{k\to \infty} \iot   {\Phi(z_\epsk)} \, W^\mathrm{el}_{1,e} ( e(u_\epsk), c_\epsk) :
    e(u_\epsk)   \,\mathrm dx  \mathrm dt = \iot \sqrt{\Phi(z)} \, \theta_{u;c;z} : e(u) \,\mathrm dx  \mathrm dt,
  \end{split}
\end{equation*}
we obtain the first assertion.Due to the convexity condition
\eqref{eqn:growthEst3}, Lemma \ref{lemma:weakConvergenceEpsilon},  Lemma
\ref{lemma:weakConverge3} and Lemma \ref{lemma:weakConverge4}, we infer 
\begin{equation*}
  \begin{split}
    \lim_{k \to \infty} & \iot ( 1- \Phi(z_\epsk)) \, | e(u) -e(u_\epsk)|^p
    \,\mathrm dx   \mathrm dt \\
    \quad & \le  \lim_{k \to \infty} \iot ( 1- \Phi(z_\epsk)) \,  
    \big(   W^\mathrm{el}_{2,e} ( e(u),c_\epsk ) -  
    W^\mathrm{el}_{2,e} ( e(u_\epsk),c_\epsk )  \big) :  (e(u) -e(u_\epsk))
    \,\mathrm dx   \mathrm dt \\
    & =  \iot ( 1- \Phi(z)) \,   W^\mathrm{el}_{2,e} ( e(u),c):e(u)  \,\mathrm dx   \mathrm dt 
    -     \iot  ( 1- \Phi(z)) \, \eta_{u;c}:e(u)
    \,\mathrm dx   \mathrm dt\\   
    & \qquad  -  \iot ( 1- \Phi(z)) \,   W^\mathrm{el}_{2,e} ( e(u),c):e(u)
    \,\mathrm dx   \mathrm dt +
    \iot  ( 1- \Phi(z)) \, \eta_{u;c}:e(u)
    \,\mathrm dx   \mathrm dt \\
    &=0.\\
  \end{split}
\end{equation*}
In consequence,
\begin{equation}
\label{eqn:con_p_h}
\big( 1 - \Phi(z_\epsk) \big)^{1/p}| e(u) -e(u_\epsk)| \to 0 \qquad \text{in } L^p(\Omega_T).
\end{equation}
We estimate 
\begin{equation*}
  \begin{split}
    \int_{\Omega_T} \Big| \big( & 1 - \Phi(z_\epsk) \big)^{1/p} e(u_\epsk)  - 
    \big( 1 - \Phi(z) \big)^{1/p}  e(u) \Big|^p \,\mathrm dx   \mathrm dt \\
    &\le \int_{\Omega_T} \Big( \big| \big( 1 - \Phi(z_\epsk) \big)^{1/p}  \big(e(u_\epsk) -
    e(u) \big) \big|  + \big|  \big(
    \big( 1 - \Phi(z_\epsk) \big)^{1/p}  -  \big( 1 - \Phi(z) \big)^{1/p}  \big)
    e(u) \big| \Big)^p \,\mathrm dx   \mathrm dt \\
    & \le  C \int_{\Omega_T}  \big( 1 - \Phi(z_\epsk) \big) \big|  \big(e(u_\epsk) -
    e(u) \big) \big|^p  \,\mathrm dx   \mathrm dt \\ 
    & \hspace{5.3cm} +
    C\int_{\Omega_T} \big| 
    \big( 1 - \Phi(z_\epsk) \big)^{1/p}  -  \big( 1 - \Phi(z) \big)^{1/p}
    \big|^p \big| e(u) \big|^p \,\mathrm dx   \mathrm dt.
  \end{split}
\end{equation*}
The first term on the right hand side converges to zero in view of \eqref{eqn:con_p_h}. 
Since $z_\epsk \to z $ a.e. in $\Omega_T$ for a suitable subsequence, we obtain 
\begin{equation*}
  \int_\Omega \big|  
  \big( 1 - \Phi(z_\epsk) \big)^{1/p}  -  \big( 1 - \Phi(z) \big)^{1/p}
  \big|^p \big| e(u) \big|^p \,\mathrm dx   \mathrm dt \to 0
\end{equation*}
by the generalized Lebesque's convergence theorem and, therefore, equation
\eqref{eqn:con_p_2} follows.
Due to \eqref{eqn:con_p_1}, \eqref{eqn:con_p_2} and  $z_\epsk \to z $ a.e. in
$\Omega$ for a subsequence $\{\eps_k\}$, we may extract a subsequence (still
denoted by $\{\eps_k\}$) such that 
\begin{align}
  \label{eqn:point1}
  \Phi(z_\epsk) \, e( u_\epsk)  &\to \Phi(z)  \, e( u)  \quad \text{ a.e. in }  \Omega_T, \\
\label{eqn:point2}
\big( 1 - \Phi(z_\epsk) \big)^{1/p} e(u_\epsk) &\to \big( 1 - \Phi(z)
\big)^{1/p} e(u) \quad \text{ a.e. in }  \Omega_T.
\end{align}
From \eqref{eqn:point1} we obtain for $\Omega_{1,T}:= \{ (t,x) \in \Omega_T: 
\Phi(z) > \frac{1}{2} \}$ 
$$ e( u_\epsk)  \to  e( u )  \quad \text{ a.e. in }  \Omega_{1,T}.   $$
Similarly, by \eqref{eqn:point2} we get  for   $\Omega_{2,T}:= \{ (t,x) \in \Omega_T: 
\Phi(z) \le \frac{1}{2} \}$ 
$$ e( u_\epsk)  \to  e( u )  \quad \text{ a.e. in }  \Omega_{2,T}.   $$
Since 
\begin{equation*}
  e( u_\epsk)   \le \sqrt{2} \sqrt{\Phi(z_\epsk)}  \, e( u_\epsk)   \qquad
  \text{in } \, \, \Big\{(t,x)  \in \Omega_T: \Phi(z_\epsk) > \frac{1}{2} \Big\}
\end{equation*}
and 
\begin{equation*}
  e( u_\epsk)   \le \sqrt[p]{2} \sqrt[p]{ \big(1-\Phi(z_\epsk) \big)}  \, e(
  u_\epsk)   \qquad \text{in } \, \,  
  \Big\{(t,x)  \in \Omega_T: \Phi(z_\epsk) \le \frac{1}{2} \Big\}
\end{equation*}
we conclude from \eqref{eqn:con_p_1}, \eqref{eqn:con_p_2} and the generalized 
Lebesque's convergence theorem 
\begin{equation}
  \label{eqn:strongpoint} 
  e(u_\epsk) \to e( u)\quad \text{in } L^p(\Omega_T) .
\end{equation}
The generalized Korn's inequality, in turn, implies 
\begin{equation*}
  \qquad \nabla u_\epsk  \to   \nabla u  \hspace{0.425cm} \ \text{ in } 
  L^p(\Omega_T) 
\end{equation*}
and therefore for a subsequence (still denoted by $\{ \epsk \}$): 
\begin{equation}
  \nabla u_\epsk \to \nabla u \quad \text{a.e. in }  \Omega_T.
\end{equation}
By similar arguments, we derive the properties of (ii)  for $\{ q_\epsk^0 \}$.
\ep
\end{proof}
\begin{lemma}
  \label{lemma:weakConvergenceWz}
  Let $\zeta \in H_+^1(\Omega)$. Then  there exists a 
  sequence $\{ q_\epsk\}$ with $ \epsk \searrow 0$ such that for a.e. $s \in [0,T]$
  \begin{equation*}   
    \begin{split}
      \int_{\Omega}  W_{,z}^\mathrm{el}(e(u(s)),c(s) ,z(s) )\zeta \,\mathrm dx  
      \le \liminf_{k \to \infty} \int_{\Omega}  W_{,z}^\mathrm{el}(e(u_\epsk(s)),c_\epsk(s) ,z_\epsk(s) )\zeta \,\mathrm dx  .
\end{split}
\end{equation*}
In addition, $W^\mathrm{el}_{,z}(e(u),c,z)$ in $L^2(0,T; L^1(\Omega))$.
\end{lemma}
\begin{proof} 
We abbreviate 
$$g(c,z):=  \Phi(z)\, C_2 \big( |c|^4 +1 \big)  +
\big( 1 -\Phi(z) \big) C_2 \big(|c|^{sp'} +1 \big) .$$
Note that due to \eqref{eq:cons_ass}
\begin{equation*}
  W_{,z}^\mathrm{el}(e(u),c,z) +g(c,z)  \ge 0 \,. 
\end{equation*}
In addition,
$$z_\epsk \to z, \qquad c_\epsk \to c \qquad  \text{and} 
\quad \nabla u_\epsk \to \nabla u \qquad a.e.~ \text{ in  }
\Omega_T$$   
for a subsequence as $\epsk \to 0$ and for a.e. $s \in [0,T]$
\begin{equation*}
  \begin{split}
    &\int_\Omega \big| g(c_\epsk(s),z_\epsk(s))  \big| \,\mathrm dx   \to 
    \int_\Omega \big| g(c(s),z(s)) \big|   \,\mathrm dx.\\
\end{split}
\end{equation*}
Therefore, we obtain the first assertion by Fatou's lemma. \\ 
Moreover, the first assertion combined with \eqref{eqn:regular4} tested by $\zeta=-1$ yields for a.e. $s \in[0,T]$:
\bes
\begin{split}
  \int_{\Omega}  \big|  W_{,z}^\mathrm{el} &(e(u(s)),c(s) ,z(s) )  + 
  g(c(s),z(s)) \big|\,\mathrm dx  \\
  & \le \liminf_{k \to \infty} \int_{\Omega}  W_{,z}^\mathrm{el}(e(u_\epsk(s)),c_\epsk(s) ,z_\epsk(s)) \,\mathrm dx
  + \int_{\Omega}   g(c(s),z(s))   \,\mathrm dx  \\
  &  \le  \liminf_{k \to \infty}   \int_{\Omega} (\alpha -\beta \partial_t z_\epsk(s) ) \,\mathrm dx
  + \int_{\Omega}   g(c(s),z(s))    \,\mathrm dx   \\
  &  \le  C       \Big( \liminf_{k \to \infty}    \|  \partial_t  z_\epsk (s)  \|_{
    L^{1}(\Omega)}  +  \|   g(c(s),z(s))  \|_{  L^{1}(\Omega)}  + 1  \Big) 
\end{split}
\ees
Hence, we obtain by Lemma \ref{lemma:boundednessEpsilon} (vii) and Fatou's lemma  
\begin{equation*}
  \begin{split}
    || W_{,z}^\mathrm{el} (e(u),c ,z )||_{L^2(0,T;L^1(\Omega))} &\le
    C \Big(  \liminf_{k\to \infty}  \|  \partial_t  z_\epsk  \|_{L^2(0,T;L^{1}(\Omega)) } 
    + \|   g(c,z)  \|_{L^2(0,T; L^{1}(\Omega))}    +1  \Big) \\
    & \le C < \infty
  \end{split}
\end{equation*}
and the second assertion follows.
\ep
\end{proof}\\[2mm] 
\begin{proof}[Proof of Theorem \ref{theorem:mainTheorem}]
  We establish items (i)-(v) of Definition
  \ref{def:weakSolutionLimit}.
  \begin{enumerate}
    \renewcommand{\labelenumi}{(\roman{enumi})}
  \item These space and regularity properties immediately follow from Lemma 
    \ref{lemma:weakConvergenceEpsilon}.
  \item
    Let $\zeta\in L^2(0,T;H^1(\Omega;\mathbb R^N))$ with 
    $\partial_t\zeta\in L^2(\Omega_T;\mathbb R^N)$ and $\zeta(T)=0$. Then,
    equation  \eqref{eqn:regular1} can be rewritten as 
    \begin{align}
      &\int_{\Omega_T}(c_{\eps_k}-c^0)\cdot\partial_t\zeta\,\mathrm dx \,  \mathrm dt
      =\int_{\Omega_T}  \mathbb{M} \nabla w_{\eps_k} :   \nabla \zeta \,  \mathrm dx
      \, \mathrm dt \, .
    \end{align}
    In view of Lemma \ref{lemma:weakConvergenceEpsilon}, we can pass to the limit and obtain 
    \eqref{eqn:limit1}. \\
    Now, let $\zeta\in L^2(0,T;H^1(\Omega;\mathbb R^N))\cap L^\infty(\Omega_T;\mathbb R^N)$.
    Integration from $t=0$ to $t=T$ of equation \eqref{eqn:regular2}  yields
    \begin{align}
      \label{eqn:auxW1}
      \begin{split}
        \int_{\Omega_T}w_\epsk & \cdot\zeta\,\mathrm dx\mathrm dt 
        =\int_{\Omega_T}\mathbb P\mathbf\Gamma\nabla c_\epsk:\nabla\zeta\,\mathrm dx\mathrm dt  \\
        &+ \int_{\Omega_T}\ (\mathbb P W_{,c}^\mathrm{ch}(c_\epsk) +\mathbb P
        W_{,c}^\mathrm{el}(e(u_\epsk),c_\epsk,z_\epsk))\cdot\zeta\,\mathrm dx\mathrm dt \, .
      \end{split}
    \end{align}
    Due to Lemma \ref{lemma:weakConvergenceEpsilon}, the growth conditions  
    for $W_{,c}^\mathrm{ch}$ and  $W_{,c}^\mathrm{el}$, Lemma 
    \ref{corollary:pointwise} and the generalized Lebesgue's convergence
    theorem, we can pass to the limit in \eqref{eqn:auxW1}: 
    \begin{align*}
      \int_{\Omega_T}w\cdot\zeta\,\mathrm dx\mathrm dt
      =\int_{\Omega_T}\mathbb P\mathbf\Gamma\nabla c:\nabla\zeta+(\mathbb P W_{,c}^\mathrm{ch}(c)
      +\mathbb P W_{,c}^\mathrm{el}(e(u),c,z))\cdot\zeta\,\mathrm dx\mathrm dt.
    \end{align*}
    Hence, we obtain for a.e. $t \in[0,T]$ equation \eqref{eqn:limit2}.
  \item This is a direct 
    consequence of Lemma \ref{lemma:weakConverge3} (ii) and (iii), 
    Lemma \ref{corollary:pointwise} and the generalized Lebesgue's convergence theorem. 
  \item From Lemma \ref{lemma:weakConvergenceEpsilon} and Lemma 
    \ref{lemma:weakConvergenceWz}, we infer the damage variational
    inequality \eqref{eqn:limit4}. The inequalities \eqref{eqn:limit4a} and \eqref{eqn:limit4b} are obvious due
    to Lemma \ref{lemma:weakConvergenceEpsilon}. 
  \item 
    Weakly semi-continuity arguments lead to
    \begin{align*}
      &\liminf_{k \to \infty }\Big(\mathcal E_\epsk(u_\epsk(t),c_\epsk(t),z_\epsk(t)) +\int_{\Omega_t}
      \alpha|\partial_t z_\epsk|+\beta |\partial_t z_\epsk|^2  +        |\nabla
      w_\epsk |^2 \, \mathrm dx \mathrm ds \Big)\\
      &\qquad\qquad\geq\mathcal E(u(t),c(t),z(t))
      +\int_{\Omega_t} \alpha|\partial_t z|+\beta |\partial_t z|^2 +    |\nabla
      w |^2  \,\mathrm dx \mathrm ds.
    \end{align*}
    Due to Lemma  \ref{corollary:pointwise} and $\lim_{ k \to \infty}
    \int_{\Omega} \epsk 
    |\nabla u^0_{\epsk}|^4   \,\mathrm dx=0$ we can pass to the limit in  
    \eqref{eqn:regular5} and obtain   \eqref{eqn:limit5}.
    \ep
  \end{enumerate}
\end{proof}

\addcontentsline{toc}{chapter}{Bibliography}{\footnotesize{\setlength{\baselineskip}{0.2 \baselineskip}
    \bibliography{references}}

\newcommand{\etalchar}[1]{$^{#1}$}
\begin{thebibliography}{BDDM07}

\bibitem[Bab11]{Bab11}
J.-F. Babadjian.
\newblock {A quasistatic evolution model for the interaction between fracture
  and damage.}
\newblock {\em Arch. Ration. Mech. Anal.}, 200(3):945--1002, 2011.

\bibitem[BB08]{BB}
E.~Bonetti and G.~Bonfanti.
\newblock Well-posedness results for a model of damage in thermoviscoelastic
  materials.
\newblock {\em Ann. Inst. H. Poincar\'e Anal. Non Lin\'eaire},
  25(6):1187--1208, 2008.

\bibitem[BCD{\etalchar{+}}02]{Bonetti02}
E.~Bonetti, P.~Colli, W.~Dreyer, G.~Gilardi, G.~Schimperna, and J.~Sprekels.
\newblock On a model for phase separation in binary alloys driven by mechanical
  effects.
\newblock {\em Physica D}, 165:48--65, 2002.

\bibitem[BDDM07]{BDDM09}
T.~B\"ohme, W.~Dreyer, F.~Duderstadt, and W.~M\"uller.
\newblock A higher gradient theory of mixtures for multi-component materials
  with numerical examples for binary alloys.
\newblock {WIAS}--{P}reprint {N}o. 1268, Weierstrass Institute for Applied
  Analysis and Stochastics, Berlin, 2007.

\bibitem[BDM07]{BDM07}
T.~B\"ohme, W.~Dreyer, and W.~M\"uller.
\newblock { Determination of stiffness and higher gradient coefficients by
  means of the embedded atom method: {A}n approach for binary alloys}.
\newblock {\em Contin. Mech. Thermodyn.}, 18:411--441, 2007.

\bibitem[BFM08]{BFM08}
B.~Bourdin, G.A. Francfort, and J.-J. Marigo.
\newblock {The variational approach to fracture.}
\newblock {\em J. Elasticity}, 91(1-3):5--148, 2008.

\bibitem[BFS13]{BFS13}
E.~Bonetti, F.~Freddi, and A.~Segatti.
\newblock {}.
\newblock {\em in preparation}, 2013.

\bibitem[BM10]{BM2010}
S.~Bartels and R.~M\"uller.
\newblock A posteriori error controlled local resolution of evolving interfaces
  for generalized {C}ahn--{H}illiard equations.
\newblock {\em Interfaces and Free Boundaries}, 12(1):45--73, 2010.

\bibitem[BMR09]{BMR09}
G.~Bouchitte, A.~Mielke, and T.~Roub\'icek.
\newblock A complete-damage problem at small strains.
\newblock {\em ZAMP Z. Angew. Math. Phys.}, 60:205--236, 2009.

\bibitem[BP05]{Pawlow}
L.~Bartkowiak and I.~Pawlow.
\newblock The {C}ahn-{H}illiard-{G}urtin system coupled with elasticity.
\newblock {\em Control and Cybernetics}, 34:1005--1043, 2005.

\bibitem[BS04]{BS04}
E.~Bonetti and G.~Schimperna.
\newblock Local existence for {F}r\'emond's model of damage in elastic
  materials.
\newblock {\em Contin. Mech. Thermodyn.}, 16(4):319--335, 2004.

\bibitem[BSS05]{BSS05}
E.~Bonetti, G.~Schimperna, and A.~Segatti.
\newblock On a doubly nonlinear model for the evolution of damaging in
  viscoelastic materials.
\newblock {\em J. of Diff. Equations}, 218(1):91--116, 2005.

\bibitem[Car86]{Car86}
A.~Carpinteri.
\newblock {\em {Mechanical damage and crack growth in concrete. Plastic
  collapse to brittle fracture}}.
\newblock {Springer}, Netherlands, 1986.

\bibitem[CFM09]{CFM09}
A.~Chambolle, G.A. Francfort, and J.-J. Marigo.
\newblock {When and how do cracks propagate?}
\newblock {\em J. Mech. Phys. Solids}, 57(9):1614--1622, 2009.

\bibitem[CFM10]{CFM10}
A.~Chambolle, G.A. Francfort, and J.-J. Marigo.
\newblock {Revisiting energy release rates in brittle fracture.}
\newblock {\em J. Nonlinear Sci.}, 20(4):395--424, 2010.

\bibitem[CMP00]{Carrive00}
M.~Carrive, A.~Miranville, and A.~Pi\'etrus.
\newblock {T}he {C}ahn-{H}illiard equation for deformable elastic media.
\newblock {\em Adv. Math. Sci. App.}, 10:539--569, 2000.

\bibitem[DPO94]{NPO94}
E.A. DeSouzaNeto, D.~Peric, and D.R.J. Owen.
\newblock { A phenomenological three-dimensional rate-independent continuum
  damage model for highly filled polymers: Formulation and computational
  aspects}.
\newblock {\em J. Mech. Phys. Solids}, 42:1533--1550, 1994.

\bibitem[EM06]{Efendiev06}
M.~A. Efendiev and A.~Mielke.
\newblock On the rate-independent limit of systems with dry friction and small
  viscosity.
\newblock {\em J. Convex Analysis}, 13:151--167, 2006.

\bibitem[FG06]{FG06}
G.A. Francfort and A.~Garroni.
\newblock {A variational view of partial brittle damage evolution.}
\newblock {\em Arch. Ration. Mech. Anal.}, 182(1):125--152, 2006.

\bibitem[FKS12]{FKS12}
A.~Fiaschi, D.~Knees, and U.~Stefanelli.
\newblock {Young-measure quasi-static damage evolution.}
\newblock {\em Arch. Ration. Mech. Anal.}, 203(2):415--453, 2012.

\bibitem[FN96]{FN96}
M.~Fr{\'e}mond and B.~Nedjar.
\newblock Damage, gradient of damage and principle of virtual power.
\newblock {\em Int. J. Solids Structures}, 33(8):1083--1103, 1996.

\bibitem[Fr{\'e}02]{Fre02}
M.~Fr{\'e}mond.
\newblock {\em {Non-smooth thermomechanics.}}
\newblock {Berlin: Springer}, 2002.

\bibitem[Gar00]{GarckeHabil}
H.~Garcke.
\newblock {\em On mathematical models for phase separation in elastically
  stressed solids}.
\newblock Habilitation thesis, University Bonn, 2000.

\bibitem[GL09]{GL09}
A.~Garroni and C.~Larsen.
\newblock {Threshold-based quasi-static brittle damage evolution.}
\newblock {\em Arch. Ration. Mech. Anal.}, 194(2):585--609, 2009.

\bibitem[GUE{\etalchar{+}}07]{Gee07}
M.G.D. Geers, R.L.J.J. Ubachs, M.~Erinc, M.A. Matin, P.J.G. Schreurs, and W.P.
  Vellinga.
\newblock Multiscale {A}nalysis of {M}icrostructura {E}volution and
  {D}egradation in {S}older {A}lloys.
\newblock {\em Internatilnal Journal for Multiscale Computational Engineering},
  5(2):93--103, 2007.

\bibitem[HK11]{HK10a}
C.~Heinemann and C.~Kraus.
\newblock Existence of weak solutions for {C}ahn-{H}illiard systems coupled
  with elasticity and damage.
\newblock {\em Adv. Math. Sci. Appl.}, 21(2):321--359, 2011.

\bibitem[HK12]{HK12}
C.~Heinemann and C.~Kraus.
\newblock Complete damage in linear elastic materials - modeling, weak
  formulation and existence results.
\newblock 2012.

\bibitem[HK13a]{HK13b}
C.~Heinemann and C.~Kraus.
\newblock 1759: Heinemann, christian; kraus, christiane a degenerating
  cahn--hilliard system coupled with complete damage processes.
\newblock 2013.

\bibitem[HK13b]{HK10b}
C.~Heinemann and C.~Kraus.
\newblock Existence results for diffuse interface models describing phase
  separation and damage.
\newblock {\em European J. Appl. Math.}, 24:179--211, 2013.

\bibitem[JL05]{LD05}
R.~Desmorat J.~Lemaitre.
\newblock {\em {E}ngineering {D}amage {M}echanics: {D}uctile, {C}reep,
  {F}atigue and {B}rittle {F}ailures}.
\newblock Springer-Verlag, Berlin, 2005.

\bibitem[KO88]{KO88}
N.~Kikuchi and J.T. Oden.
\newblock {\em {Contact problems in elasticity: A study of variational
  inequalities and finite element methods.}}
\newblock {Philadelphia, PA: SIAM}, 1988.

\bibitem[KRZ11]{KRZ11}
D.~Knees, R.~Rossi, and C.~Zanini.
\newblock {\em A Vanishing Viscosity Approach to a Rate-independent Damage
  Model}.
\newblock Preprint No. 1633. WIAS, 2011.

\bibitem[LT11]{LT11}
G.~Lazzaroni and R.~Toader.
\newblock {A model for crack propagation based on viscous approximation.}
\newblock {\em Math. Models Methods Appl. Sci.}, 21(10):2019--2047, 2011.

\bibitem[Mer05]{Mer05}
T.~Merkle.
\newblock The {C}ahn-{L}arch\'e system: a model for spinodal decomposition in
  eutectic solder; modelling, analysis and simulation.
\newblock Phd-thesis, Universit\"at Stuttgart, Stuttgart, 2005.

\bibitem[Mie95]{Mie95}
C.~Miehe.
\newblock {Discontinuous and continuous damage evolution in Ogden-type
  large-strain elastic materials}.
\newblock {\em Eur. J. Mech.}, 14:697--720, 1995.

\bibitem[Mie05]{Mielke05}
A.~Mielke.
\newblock Evolution in rate-independent systems.
\newblock {\em Handbook of Differential Equations: Evolutionary Equations},
  2:461--559, 2005.

\bibitem[Mie11]{Mie11}
A.~Mielke.
\newblock {Complete-damage evolution based on energies and stresses.}
\newblock {\em Discrete Contin. Dyn. Syst., Ser. S}, 4(2):423--439, 2011.

\bibitem[MK00]{MK00}
C.~Miehe and J.~Keck.
\newblock {Superimposed finite elastic-viscoelastic-plastoelastic stress
  response with damage in filled rubbery polymers. Experiments, modelling and
  algorithmic implementation}.
\newblock {\em J. Mech. Phys. Solids 48}, 48:323--365, 2000.

\bibitem[MR06]{Mielke06}
A.~Mielke and T.~Roub\'icek.
\newblock Rate-independent damage processes in nonlinear elasticity.
\newblock {\em Mathematical Models and Methods in Applied Sciences},
  16:177--209, 2006.

\bibitem[MRZ10]{Mielke10}
A.~Mielke, T.~Roub\'icek, and J.~Zeman.
\newblock Complete {D}amage in elastic and viscoelastic media.
\newblock {\em Comput. Methods Appl. Mech. Engrg}, 199:1242--1253, 2010.

\bibitem[MS11]{MS01}
A.~Menzel and P.~Steinmann.
\newblock {A theoretical and computational framework for anisotropic continuum
  damage mechanics at large strains}.
\newblock {\em Int. J. Solids Struct.}, 38:9505--9523, 2011.

\bibitem[MT99]{Mielke99}
A.~Mielke and F.~Theil.
\newblock A mathematical model for rate-independent phase transformations with
  hysteresis.
\newblock In R.~Balean H.-D.~Alber and R.~Farwig, editors, {\em Models of
  Continuum Mechanics in Analysis and Engineering}, pages 117--129, Aachen,
  1999. Shaker Verlag.

\bibitem[MT10]{MT10}
A.~Mielke and M.~Thomas.
\newblock Damage of nonlinearly elastic materials at small strain ---
  {E}xistence and regularity results.
\newblock {\em ZAMM Z. Angew. Math. Mech.}, 90:88--112, 2010.

\bibitem[Neg10]{Neg10}
M.~Negri.
\newblock {From rate-dependent to rate-independent brittle crack propagation.}
\newblock {\em J. Elasticity}, 98(2):159--187, 2010.

\bibitem[Nit81]{Nit81}
J.A. Nitsche.
\newblock {On Korn's second inequality.}
\newblock {\em RAIRO, Anal. Num\'er.}, 15:237--248, 1981.

\bibitem[PZ08]{Pawlow08}
I.~Pawlow and W.~M. Zaj\c{a}czkowski.
\newblock Measure-valued solutions of a heterogeneous {C}ahn-{H}illiard system
  in elastic solids.
\newblock {\em Colloquium Mathematicum}, 112 No.2, 2008.

\bibitem[Rou10]{Roubicek10}
T.~Roub\'icek.
\newblock Thermodynamics of rate-independent processes in viscous solids at
  small strains.
\newblock {\em SIAM J. Math. Anal.}, 42 No. 1, 2010.

\bibitem[RR12]{RR12}
E.~Rocca and R.~Rossi.
\newblock A degenerating {PDE} system for phase transitions and damage.
\newblock {\em arXiv:1205.3578v1}, 2012.

\bibitem[Sim86]{Simon}
J.~Simon.
\newblock Compact sets in the space ${L}^p(0,{T};{B})$.
\newblock {\em Annali di Matematica Pura ed Applicata}, 146:65--96, 1986.

\bibitem[SP12]{SP12}
G.~Schimperna and I.~Paw{\l}ow.
\newblock {On a class of Cahn-Hilliard models with nonlinear diffusion}.
\newblock {\em arXiv:1106.1581}, 2012.

\bibitem[SP13]{SP13}
G.~Schimperna and I.~Paw{\l}ow.
\newblock {A Cahn-Hilliard equation with singular diffusion}.
\newblock {\em J. Differ. Equations}, 254(2):779--803, 2013.

\bibitem[VSL11]{VSL11}
G.~Z. Voyiadjis, A.~Shojaei, and G.~Li.
\newblock {A thermodynamic consistent damage and healing model for self healing
  materials.}
\newblock {\em Int. J. Plast.}, 27(7):1025--1044, 2011.

\bibitem[Wei01]{Wei02}
U.~Weikard.
\newblock Numerische {L}\"osungen der {C}ahn-{H}illiard-{G}leichung und der
  {C}ahn-{L}arch\'e-{G}leichung.
\newblock Phd-thesis, Universit\"at Bonn, Bonn, 2001.

\end{thebibliography}
  \bibliographystyle{alpha}}
\end{document}